\documentclass[12pt]{amsart}
\pdfoutput=1
\usepackage[margin=1in]{geometry}
\usepackage[utf8]{inputenc}
\usepackage{graphicx}
\usepackage{amsmath,amssymb}
\usepackage{amsthm}
\usepackage{color}
\usepackage{tikz, cite}
\usepackage[utf8]{inputenc}
\usepackage{pgfplots}
\pgfplotsset{compat=1.18}
\usetikzlibrary{fillbetween}
\allowdisplaybreaks

\makeatletter
\def\l@section{\@tocline{1}{12pt plus2pt}{0pt}{}{\bfseries}}
\def\l@subsection{\@tocline{2}{0pt}{2pc}{2pc}{}}
\makeatother
\makeatletter
\def\subsection{\@startsection{subsection}{2}{\z@}%
	{-3.25ex\@plus -1ex \@minus -.2ex}%
	{1.5ex \@plus .2ex}%
	{\normalfont\bfseries\boldmath}}
\def\subsubsection{\@startsection{subsubsection}{3}%
	\z@{.5\linespacing\@plus.7\linespacing}{-.5em}%
	{\normalfont\bfseries\boldmath}}
\renewcommand\paragraph{\@startsection{paragraph}{4}{\z@}%
	{3.25ex \@plus1ex \@minus.2ex}%
	{-1em}%
	{\normalfont\normalsize\bfseries}}
\makeatother

\theoremstyle{plain}
\newtheorem{thm}{Theorem}[section]

\newtheorem{lem}[thm]{Lemma}
\newtheorem{prop}[thm]{Proposition}

\theoremstyle{definition}
\newtheorem{defn}[thm]{Definition}

\theoremstyle{remark}
\newtheorem{rem}[thm]{Remark}

\theoremstyle{plain}

\numberwithin{equation}{section}

\theoremstyle{plain} 
\newcommand{\thistheoremname}{}
\newtheorem{genericthm}[thm]{\thistheoremname}

  \newtheorem*{genericthm*}{\thistheoremname}
\newenvironment{namedthm*}[1]
  {\renewcommand{\thistheoremname}{#1}%
   \begin{genericthm*}}
  {\end{genericthm*}}

\newcommand{\R}{{\mathbb R}}

\newcommand{\N}{{\mathbb N}}

\newcommand{\Z}{{\mathbb Z}}
\newcommand{\calT}{{\mathcal T}}

\newcommand{\calS}{{\mathcal S}}
\newcommand{\supp}{{\textnormal{supp}}}
\newtheorem{thmA}{Theorem}

\makeatletter
\newcommand{\vast}{\bBigg@{4}}
\newcommand{\Vast}{\bBigg@{5}}
\makeatother

\def\udot#1{\ifmmode\oalign{$#1$\crcr\hidewidth.\hidewidth
    }\else\oalign{#1\crcr\hidewidth.\hidewidth}\fi}

\def\R{\mathbb{R}}
\def\Z{\mathbb{Z}}

\def\beq{\begin{equation}}
\def\eeq{\end{equation}}
\makeatletter
\newcommand{\doublewidetilde}[1]{{%
  \mathpalette\double@widetilde{#1}%
}}
\newcommand{\double@widetilde}[2]{%
  \sbox\z@{$\m@th#1\widetilde{#2}$}%
  \ht\z@=.9\ht\z@
  \widetilde{\box\z@}%
}
\makeatother

\def\one{\mbox{1\hspace{-4.25pt}\fontsize{12}{14.4}\selectfont\textrm{1}}}

\usepackage{tikz}

\makeatletter
\def\@makefnmark{%
  \leavevmode
  \raise.9ex\hbox{\fontsize\sf@size\z@\normalfont\tiny\@thefnmark}}
\makeatother

\begin{document}
	
\title[]{Hardy space estimates for the parabolic maximal operator below $p=1$}

\author{Yongsheng Han}
\address{(Yongsheng Han) Department of Mathematics and Statistics\\
         Auburn University\\
         Auburn, Alabama, U.S.A, 36849}
\email{hanyong@auburn.edu}

\author{Bingyang Hu}
\address{(Bingyang Hu) Department of Mathematics and Statistics\\
        Auburn University\\
        Auburn, Alabama, U.S.A, 36849}
\email{bzh0108@auburn.edu}

\begin{abstract}
We study the parabolic maximal operator $M_{\textnormal{par}}$ along the
moment curve $(t,t^2)$. In 1988, Christ proved that
$M_{\textnormal{par}}$ maps the parabolic Hardy space
$H_{\textnormal{par}}^1(\R^2)$ into $L^{1,\infty}(\R^2)$.
We show that this result is sharp at $p=1$: for every $0<p<1$,
the natural extension from $H_{\textnormal{par}}^p(\R^2)$ to
$L^{p,\infty}(\R^2)$ fails even for the corresponding single-scale
operator. We then introduce a curvature-adapted modified Hardy space
$H_{\textnormal{par}}^{p,*}(\R^2)$ and a weak tendril space
$\calT^{p,\infty}(\R^2)$, and prove that
$$
M_{\textnormal{par}}:
H_{\textnormal{par}}^{p,*}(\R^2)
\longrightarrow
\calT^{p,\infty}(\R^2), \qquad 0<p<1,
$$
is bounded. At $p=1$, these spaces recover those in Christ's theorem:
$H_{\textnormal{par}}^{1,*}(\R^2)=H_{\textnormal{par}}^1(\R^2)$ and
$\calT^{1,\infty}(\R^2)=L^{1,\infty}(\R^2)$. Thus, our result provides a natural
extension of Christ's work to the range $0<p<1$.
\end{abstract}

\date{\today}



\maketitle
\tableofcontents

\section{Introduction}
Let $\varphi$ be a positive bump function on $\R$ such that $\varphi \le 1$ with $\varphi \equiv 1$ on $(1, 2)$ and $\supp \; \varphi \subseteq (1/2, 5/2)$.
In this paper, we are interested in the following maximal operator along the moment curve $(t, t^2)$:
\begin{align} \label{20260625avgop}
M_{\textnormal{par}}f(x)
&:=\sup_{r>0} \left| A_r f(x) \right| \nonumber  \\ 
&:=\sup_{r>0} \left| \int_{\R} f(x_1-t, x_2-t^2) \varphi \left(\frac{t}{r} \right) \frac{dt}{r} \right|. 
\end{align}
In 1988, Christ in his seminal work \cite{Christ1988} proved the following deep theorem. 
\begin{thmA}[{\cite[Theorem~3]{Christ1988}}] \label{20260625thm01}
 $M_{\textnormal{par}}$ is bounded from
$H_{\textnormal{par}}^1(\mathbb R^2)$ to $L^{1,\infty}(\mathbb R^2)$.
\end{thmA}

Here $H_{\textnormal{par}}^1(\mathbb R^2)$ denotes the \emph{parabolic Hardy space}
associated with the anisotropic dilations
$$
\delta_r(x_1,x_2)=(r x_1, r^2 x_2), \qquad r>0.
$$
It is clear that the homogeneous dimension of the above dilation is $3$. Following Christ's definition, $H_{\textnormal{par}}^1(\mathbb R^2)$ is the
subspace of $L^1(\mathbb R^2)$ consisting of all functions $f$ admitting an
atomic representation
$$
f(x)=\sum_Q \lambda_Q a_Q(x),
$$
where
\begin{enumerate}
    \item a \emph{parabolic box} $Q$ is a rectangle in $\mathbb R^2$ of the form 
    $$
    Q=Q(z, \rho):=z+[-\rho, \rho] \times [-\rho^2, \rho^2]
    $$
    for $z=(z_1, z_2) \in \R^2$ and $\rho>0$; we
    write its side length by $\rho_Q=\rho$;

    \item $a_Q$ is a $Q$-atom, meaning that
    $$
    \operatorname{supp} a_Q\subseteq Q,\qquad
    \|a_Q\|_{L^\infty(\mathbb R^2)}\le |Q|^{-1},
    \qquad
    \int_{\mathbb R^2} a_Q(x)\,dx=0;
    $$

    \item the coefficient sequence satisfies $\{\lambda_Q\}\in \ell^1$.
\end{enumerate}
The $H_{\textnormal{par}}^1(\R^2)$ norm is given by
$$
\|f\|_{H_{\textnormal{par}}^1(\R^2)}
=
\inf \sum_Q |\lambda_Q|,
$$
where the infimum is taken over all such atomic representations. Here, the rectangles $Q$ in the atomic representation are \emph{not} required to be disjoint.

\begin{rem}
The \emph{curvature} of the parabola $(t,t^2)$ plays an essential role in
Theorem~\ref{20260625thm01}. More broadly, this reflects a fundamental
and longstanding theme in harmonic analysis: the role of curvature in
the $L^p$ theory of curved operators; see, for instance, the classical and
influential survey \cite{SteinWainger1978}. We mention several representative directions.

\begin{enumerate}
    \item[(i)] \emph{Maximal and averaging operators.}
    The study of maximal and averaging operators associated with
    parabolic dilations, smooth curves, and curved hypersurfaces goes
    back to the works of Calder\'on and Torchinsky, Nagel, Rivi\`ere,
    Stein, and Wainger
    \cite{CalderonTorchinsky1975,
    NagelRiviereWainger1976Maximal,
    SteinWainger1976}.
    Subsequent developments include results for averages over convex
    hypersurfaces, $L^p$ regularity of averages over curves, and sharp
    bounds for more general curved maximal operators; see
    \cite{NagelSeegerWainger1993,
    PramanikSeeger2007,
    BeltranGuoHickmanSeeger2025}.

    \item[(ii)] \emph{Hilbert transforms and singular integrals along
    curves.}
    The $L^p$ theory of Hilbert transforms along curves and surfaces was
    initiated and systematically developed in
    \cite{NagelRiviereWainger1974,
    NagelRiviereWainger1976Hilbert,
    NagelWainger1976,
    NagelWainger1977}.
    More recent work has treated maximal families of Hilbert transforms
    along nonflat homogeneous curves; see, for example,
    \cite{GuoRoosSeegerYung2020}.

    \item[(iii)] \emph{Multilinear extensions.}
    Curvature and the corresponding nondegeneracy conditions also play
    an important role in multilinear harmonic analysis. Representative
    examples include multilinear oscillatory integrals
    \cite{ChristLiTaoThiele2005}, bilinear Hilbert transforms and
    bilinear maximal operators along curved polynomial trajectories
    \cite{Li2013, LiXiao2016}, and multilinear averaging and
    lacunary maximal operators associated with curved hypersurfaces
    \cite{ChoLeeShuin2024}.
\end{enumerate}

Much of the classical linear theory above concerns strong-type estimates
in the range $p>1$, whereas the endpoint $p=1$ is substantially more
delicate, as already reflected in the weak-type Hardy-space formulation
of Theorem~\ref{20260625thm01}. A stronger estimate at this endpoint is available in the setting of
the product Hardy space $H_{\textnormal{prod}}^1(\mathbb R^2)$.
More precisely, Seeger and Tao
\cite[Theorem~1.4]{SeegerTao2001} proved that
$M_{\textnormal{par}}$ maps
$H_{\textnormal{prod}}^1(\mathbb R^2)$ boundedly into
the Lorentz space $L^{1,2}(\mathbb R^2)$.
Since $H_{\textnormal{prod}}^1(\mathbb R^2)$ is a proper
subspace of $H_{\textnormal{par}}^1(\mathbb R^2)$,
this result can be viewed as a complement of Theorem~\ref{20260625thm01}.
\end{rem}

The \emph{goal} of this paper is twofold:
\begin{enumerate}
    \item First, we show that Theorem~\ref{20260625thm01} is sharp at
    $p=1$. More precisely, if $H_{\textnormal{par}}^1(\R^2)$ is naturally
    extended to its $H^p$ counterpart $H_{\textnormal{par}}^p(\R^2)$,
    as in Definition~\ref{def:Hparp}, then, for every $0<p<1$,
    $M_{\textnormal{par}}$ is not bounded from
    $H_{\textnormal{par}}^p(\R^2)$ to $L^{p,\infty}(\R^2)$;

    \item Second, we develop an atomic parabolic $H^p$ theory based on
$(p,\infty)$-atoms and adapted to the curvature and scaling inherent in
$M_{\textnormal{par}}$, thereby obtaining an extension of Christ's result
Theorem~\ref{20260625thm01} to the range $0<p<1$. 
\end{enumerate}

\medskip

 We begin by extending the definition of
\(H_{\textnormal{par}}^1(\mathbb R^2)\) to \(H_{\textnormal{par}}^p(\mathbb R^2)\),
\(0<p\le 1\), following the standard atomic formulation from the classical
theory of Hardy spaces.

\begin{defn} [Parabolic $(p, \infty)$-atom] \label{atomdefn}
Let $0<p \le 1$, and choose a nonnegative integer
\begin{equation} \label{20260625eq01}
N_p \ge \left\lfloor 3\left(\frac{1}{p}-1 \right) \right\rfloor.
\end{equation} 
A \emph{parabolic \((p, \infty)\)-atom} is a bounded measurable function \(a\) for which there exists a parabolic box \(Q\) such that
\begin{align*}
&\supp \; a \subset Q, \\\
&\|a\|_{L^\infty(\R^2)}\le |Q|^{-\frac{1}{p}} \simeq \rho_Q^{-\frac{3}{p}}, \\
&\int_{\R^2} x^\beta a(x)\,dx=0, \qquad |\beta|\le N_p. 
\end{align*}
Here $\beta=(\beta_1,\beta_2)\in\N^2$, $|\beta|=\beta_1+\beta_2$, and $x^\beta=x_1^{\beta_1}x_2^{\beta_2}$.
\end{defn}

\begin{defn}[Atomic parabolic Hardy space]\label{def:Hparp}
For $0<p\le 1$, define $H^p_{\textnormal{par}}(\mathbb R^2)$ to be the space of all
tempered distributions $f\in L_{\textrm{loc}}^1(\mathbb R^2)$ admitting a representation
\begin{equation}\label{eq:Hparp-rep}
        f=\sum_Q \lambda_Q a_Q
        \quad\text{in }  L_{\textnormal{loc}}^1(\mathbb R^2),
\end{equation}
where each $a_Q$ is a parabolic $(p, \infty)$-atom, and
$$
        \sum_Q |\lambda_Q|^p<\infty.
$$
Set
$$
        \|f\|_{H^p_{\textnormal{par}}}^p
        :=
        \inf \sum_Q |\lambda_Q|^p,
$$
where the infimum is taken over all representations \eqref{eq:Hparp-rep}.
\end{defn}

\begin{rem}
Recall that in the classical $H^p$ theory on $\R^2$ with standard Euclidean dilation, $f$ is assumed to belong to $\calS'(\R^2)$. Here, for simplicity and the well-defineness of the operator \eqref{20260625avgop}, we restricted our attention to $L_{\textrm{loc}}^1(\mathbb R^2)$.
\end{rem}

\begin{rem}
We remark that for $0<p\leq 1$, the parabolic Hardy space
$H_{\textnormal{par}}^p(\mathbb R^2)$ can also be characterized
by parabolic $(p,2)$-atoms or Littlewood--Paley square functions;
see, for instance,
\cite{Calderon1977, CoifmanWeiss1977, LatterUchiyama1979,
Bownik2003, Sato2018}.
Here we recall that a parabolic $(p,2)$-atom satisfies the same
support and moment conditions as a $(p,\infty)$-atom in
Definition~\ref{atomdefn}, with the $L^\infty$ size condition
replaced by
$\|a\|_{L^2(\mathbb R^2)}\leq |Q|^{1/2-1/p}$.

\vspace{0.1cm}

Littlewood--Paley theory plays a central role in the Hardy-space
theory of anisotropic Calder\'on--Zygmund operators satisfying
suitable smoothness and cancellation conditions. The corresponding
boundedness results, however, do \emph{not} directly apply to the curved
maximal operator considered here. Indeed, Theorem~\ref{mainthm01} below
shows that the estimate from $H_{\textnormal{par}}^p(\mathbb R^2)$
to $L^{p,\infty}(\mathbb R^2)$ fails for every $0<p<1$, even for
the single-scale averaging operator. This obstruction persists
regardless of whether $H_{\textnormal{par}}^p(\mathbb R^2)$ is
described by atoms or by Littlewood--Paley square functions.
\end{rem}

To this end, for $0<p<\infty$, we recall that the weak $L^p$ space $L^{p,\infty}(\mathbb R^2)$ is the space
of all measurable functions $f$ on $\mathbb R^2$ such that
$$
\|f\|_{L^{p,\infty}(\mathbb R^2)}:=\sup_{\alpha>0} \alpha\, \big|\{x\in\mathbb R^2: |f(x)|>\alpha\}\big|^{1/p}<\infty.
$$
Note that in the regime we are interested, namely, $0<p<1$, the above quantity becomes a quasi-norm.

\medskip

We are now ready to state the first main result of this paper.

\begin{thm}\label{mainthm00}
For every $0<p<1$, the parabolic maximal operator $M_{\textnormal{par}}$ is not bounded from
$H_{\textnormal{par}}^p(\mathbb R^2)$ to $L^{p,\infty}(\mathbb R^2)$.
\end{thm}

Theorem~\ref{mainthm00} follows immediately from the following stronger single-scale result.

\begin{thm}\label{mainthm01}
Recall the single-scale operator
\begin{equation} \label{20260625singlescale}
Af(x)=A_1f(x)
:=\int_{\mathbb R} f(x_1-t,x_2-t^2)\varphi(t)\,dt.
\end{equation} 
Then, for every $0<p<1$, $A$ is not bounded from
$H_{\textnormal{par}}^p(\mathbb R^2)$ to $L^{p,\infty}(\mathbb R^2)$.
\end{thm}

Our next goal is therefore to develop an appropriate atomic $H^p$ theory of
parabolic Hardy spaces for $0<p<1$ for the parabolic $(p, \infty)$-atoms, within which Christ's result admits a
natural extension. We begin with a brief overview.

\begin{enumerate}
    \item[$\bullet$] First, motivated by Theorem~\ref{mainthm01}, we observe
    that each parabolic box $Q$ carries a natural curvature-induced cost;
    see \eqref{20260704eq01}. This leads us to modify the definition of
    $H_{\textnormal{par}}^p(\R^2)$ and introduce the modified space
    $H_{\textnormal{par}}^{p,*}(\R^2)$; see
    Definition~\ref{Hparpstar}. When $p=1$, this modified space reduces
    to the original parabolic Hardy space
    $H_{\textnormal{par}}^1(\R^2)$.

    \vspace{0.1cm}

    \item[$\bullet$] Second, Proposition~\ref{20260709prop01}, which may be
    viewed as a large-scale counterpart of Theorem~\ref{mainthm01}, shows
    that modifying the domain space alone does \emph{not} suffice. This suggests
    that the target space $L^{p,\infty}(\R^2)$ must also be replaced by a
    space adapted to the underlying parabolic geometry. We call this new
    function space the \emph{weak tendril space} $\calT^{p,\infty}(\R^2)$; see
    Definition~\ref{20260714defn01}. This space also recovers the original
    weak-type space at $p=1$, that is, 
    $$
        \calT^{1,\infty}(\R^2)
        =
        L^{1,\infty}(\R^2);
    $$
    see, Proposition~\ref{20260710prop01}.
\end{enumerate}

Our second main result is the following.

\begin{thm}\label{mainwholeX}
Let $0<p<1$. Then $M_{\textnormal{par}}$ is bounded from
$H_{\textnormal{par}}^{p,*}(\R^2)$ to
$\calT^{p,\infty}(\R^2)$.
\end{thm}

Theorem~\ref{mainwholeX} may be viewed as completing a classical
Hardy-space pattern in the presence of curvature. Let
$$
        M_\Phi f(x)
        :=
        \sup_{r>0}|f*\Phi_r(x)|,
        \qquad
        \Phi_r(x):=r^{-2}\Phi(x/r),
$$
be a standard smooth maximal operator characterizing the classical
Hardy spaces on $\R^2$, where $\Phi$ is a smooth bump in $\R^2$. The comparison can be summarized as follows:
$$
\begin{array}{c|c|c}
 & \textnormal{classical $H^p$ theory}
 & \textnormal{parabolic $H^p$ theory induced by $(p, \infty)$-atoms}
 \\[2mm] \hline
  \rule{0pt}{3ex}p=1
 &
 M_\Phi:H^1(\R^2)\longrightarrow L^1(\R^2)
 &
 M_{\textnormal{par}}:
 H_{\textnormal{par}}^1(\R^2)
 \longrightarrow
 L^{1,\infty}(\R^2)
 \quad\textnormal{(Theorem \ref{20260625thm01})}
 \\[3mm]
 0<p<1
 &
 M_\Phi:H^p(\R^2)\longrightarrow L^p(\R^2)
 & M_{\textnormal{par}}:
 H_{\textnormal{par}}^{p,*}(\R^2)
 \longrightarrow
 \calT^{p,\infty}(\R^2)
 \quad\textnormal{(Theorem~\ref{mainwholeX})}
\end{array}
$$
Thus, Theorem \ref{20260625thm01} established the $p=1$ part of the parabolic picture, while
Theorem~\ref{mainwholeX} supplies its curvature-adapted counterpart in
the range $0<p<1$. We now describe the sequence of intermediate attempts that led to Theorem~\ref{mainwholeX}.

\vspace{0.05cm} 

\begin{center}
\renewcommand{\arraystretch}{1.6}
\begin{tabular}{@{}p{0.48\textwidth}p{0.42\textwidth}@{}}
\textbf{Mapping statements for $0<p<1$}
&
\textbf{Results in this paper}
\\
\hline
$H_{\textnormal{par}}^p(\mathbb R^2)
\longrightarrow L^{p,\infty}(\mathbb R^2)$
&
Failure: Theorems~1.6--1.7
\\
\hline
$H_{\textnormal{par}}^{p,\ast}(\mathbb R^2)
\longrightarrow L^{p,\infty}(\mathbb R^2)$
&
Failure: Proposition~3.2
\\
\hline
$H_{\textnormal{par}}^p(\mathbb R^2)
\longrightarrow \mathcal T^{p,\infty}(\mathbb R^2)$
&
Failure: Proposition~4.8
\\
\hline
$H_{\textnormal{par}}^{p,\ast}(\mathbb R^2)
\longrightarrow \mathcal T^{p,\infty}(\mathbb R^2)$
&
Bounded: Theorem~1.8
\end{tabular}
\end{center}
This comparison shows that, among these four natural choices,
modifying only the source space or only the target space does not
suffice. The bounded estimate in Theorem~1.8 is obtained only after
introducing both the curvature-adapted Hardy space and the weak tendril
space.

\begin{rem}
At the endpoint $p=1$, Christ's argument has a fundamentally
different structure. The proof of
Theorem~\ref{20260625thm01} is based on a delicate stopping-time
decomposition of the atoms, the construction of an
appropriate exceptional set, and an $L^2$ estimate away from that
set. This organization is necessary because
$L^{1,\infty}(\mathbb R^2)$ does \emph{not} admit a uniform
$\ell^1$-summation estimate for arbitrary finite families, with a
constant independent of the number of terms. 

We emphasize that Theorem~\ref{mainwholeX} does \emph{not} include
Christ's result at $p=1$. Its proof relies essentially on the
assumption $p<1$, but uses a different summation mechanism. Indeed,
for $0<p<1$, the weak space $L^{p,\infty}(\mathbb R^2)$ satisfies
the standard $p$-subadditivity estimate
$$
    \left\|\sum_{j=1}^N f_j\right\|_{L^{p,\infty}}^p
    \lesssim_p
    \sum_{j=1}^N
    \|f_j\|_{L^{p,\infty}}^p,
$$
where the implicit constant is independent of $N$; see
\cite[Theorem~6.1]{Kalton1980} and
\cite{SteinTaiblesonWeiss1981}.
However, the counterexamples underlying
Theorems~\ref{mainthm00} and~\ref{mainthm01} show that the standard
estimate
$$
M_{\mathrm{par}}:
H_{\mathrm{par}}^p(\mathbb R^2)
\longrightarrow L^{p,\infty}(\mathbb R^2)
$$
already fails at the level of a single parabolic atom. Thus, the
$p$-subadditivity of $L^{p,\infty}(\R^2)$ alone cannot resolve this
obstruction.

Our modified Hardy space and weak tendril space are designed to
capture the scale dependence and the parabolic geometry appearing
in these single-atom estimates. The curvature of the parabola first
provides a tendril covering of each level set of
$M_{\mathrm{par}}a_Q$, yielding a uniform estimate for each atom in
terms of the scale cost $w_p(Q)$; see
Lemma~\ref{20260713lem01} and Proposition~\ref{20260714prop01}.
These single-atom estimates are then combined through a
corresponding generalized triangle inequality in
$\mathcal T^{p,\infty}(\mathbb R^2)$. The fractional covering
formulation of the tendril content allows us to establish this
inequality in Lemma~\ref{20260714prop22}, which provides the required
control of the sum of the atomic contributions. Theorem~\ref{mainwholeX}
then follows from the more precise statement in
Theorem~\ref{mainwhole}.
\end{rem}

\vspace{0.1cm}

The rest of the paper is organized as follows. In
Section~\ref{Sec02}, we prove the sharpness of Christ's theorem by
constructing a parabolic $(p,\infty)$-atom for which the corresponding
weak $L^p$ estimate fails. In Section~\ref{Sec03}, we introduce the
modified parabolic Hardy space $H_{\textnormal{par}}^{p,*}(\R^2)$ and
show that modifying the domain space alone does not yield the desired
extension. We then develop the theory of weak tendril spaces and prove
Theorem~\ref{mainwholeX} in Section~\ref{Sec04}. Finally,
Section~\ref{Sec05} provides some further remarks on real interpolation
of the modified parabolic Hardy spaces and weak tendril spaces.

Throughout the paper, for nonnegative quantities $a$ and $b$, we write
$a\lesssim b$ if $a\leq Cb$ for some constant $C>0$ independent of $a$ and $b$. We write
$a\simeq b$ if both $a\lesssim b$ and $b\lesssim a$ hold. We also write
$x=(x_1,x_2)$, $X=(X_1,X_2)$, $y=(y_1,y_2)$, $Y=(Y_1,Y_2)$, and
$z=(z_1,z_2)$ for vectors in $\R^2$.
\\
\noindent{\bf Acknowledgement.}
The authors are deeply grateful to Andreas Seeger for taking the time to read an earlier version of the manuscript and for his helpful comments and suggestions. In particular, he pointed out an error in the discussion of the $(p,\infty)$- and $(p,2)$-atomic characterizations of parabolic Hardy spaces in Remark~1.1 of that version. He also brought to their attention the generalized triangle inequality in $L^{p,\infty}(\R^2)$ for $0<p<1$ and provided relevant references. These comments helped clarify the different difficulties in the cases $p<1$ and $p=1$, as discussed in Remark~1.9. They also thank him for drawing their attention to his joint work with Terence Tao on the $H_{\textnormal{prod}}^1(\R^2)\to L^{1,2}(\R^2)$ estimate, which complements Christ's parabolic Hardy-space estimate. His suggestions have substantially improved the accuracy and presentation of the paper. The authors would also like to thank Ji Li for bringing this problem to their attention during his visit to Auburn University in the fall of 2023. The second author was supported by NSF grant DMS-2555999 and by the Simons Foundation Travel Support for Mathematicians program under award MPS-TSM-00007213.

\bigskip

 \section{Sharpness of Christ's estimate: Proof of Theorem \ref{mainthm01}} \label{Sec02}

We begin with an overview of the proof. Fix $0<p<1$, and recall that $N_p$ denotes the positive integer defined in \eqref{20260625eq01}. For $0<\rho<1$, write
$$
Q_\rho:=Q(0, \rho)=[-\rho, \rho] \times [-\rho^2, \rho^2]. 
$$
Throughout the proof, $\rho>0$ will be regarded as a sufficiently small parameter. Our goal is to construct a parabolic $(p, \infty)$-atom ${\mathfrak a}_{Q_\rho}$ and a measurable set $E_\rho \subseteq \R^2$, satisfying 
\begin{align}
        |E_\rho|&\gtrsim \rho, \label{20260625goal01}\\
        \left|A {\mathfrak a}_{Q_\rho}(x) \right| &\gtrsim \rho^{2-\frac{3}{p}},
        \qquad x\in E_\rho. \label{20260625goal02}
\end{align}
Temporarily assuming both \eqref{20260625goal01} and \eqref{20260625goal02}, then 
\begin{align} \label{20260704eq01}
\left\|A{\mathfrak a}_{Q_\rho} \right\|_{L^{p, \infty}(\R^2)}
&=\sup_{\alpha>0} \alpha \left| \left\{ x\in \R^2: \left|A{\mathfrak a}_{Q_\rho}(x) \right|>\alpha \right\} \right|^{\frac{1}{p}}  \nonumber \\
& \ge \rho^{2-\frac{3}{p}} \cdot \left|E_\rho \right|^{\frac{1}{p}} \nonumber \\ 
& \ge \rho^{2-\frac{3}{p}} \cdot \rho^{\frac{1}{p}} \nonumber  \\
&= \rho^{2-\frac{2}{p}}.
\end{align}
Since $0<p<1$, we have
$$
\left\|A{\mathfrak a}_{Q_\rho} \right\|_{L^{p, \infty}(\R^2)} \to \infty \qquad \textrm{as} \qquad \rho \to 0, 
$$
which concludes the proof of Theorem \ref{mainthm01}. 

\begin{rem}
The above computation shows that $p=1$ is the threshold case: when $0<p<1$, the construction gives a counterexample, while at $p=1$ it no longer does so. This is consistent with Christ's theorem
$$
M_{\varphi}:H^1_{\textnormal{par}}(\mathbb R^2)\to L^{1,\infty}(\mathbb R^2),
$$
and shows that this endpoint result is sharp in the scale of atomic parabolic Hardy spaces.
\end{rem}

It remains to prove \eqref{20260625goal01} and \eqref{20260625goal02}. Before doing so, we first explain the motivation behind the construction.

\medskip 

\noindent {$\diamond$ \bf Heuristic.} The curvature of $(t, t^2)$ in the single-scale
operator $A$ is crucial here. Let $0<\eta<10^{-3}$, and consider the
horizontal segment
$$
\Sigma_{\rho}:=\{(u,0): |u|<\eta\rho\}\subset Q_\rho.
$$
Now take
\begin{equation} \label{20260625comp01}
x=(u,0)+(\tau,\tau^2),
\qquad |u|<\eta\rho,\quad
\tau\in \supp \; \varphi\subset \left(\frac12,\frac52\right).
\end{equation} 
Then for any parabolic $(p, \infty)$-atom $a_{Q_\rho}$, one has 
$$
Aa_{Q_\rho}(x)=\int_{\mathbb R}
a_{Q_\rho}(u+\tau-t,\tau^2-t^2)\varphi(t)\,dt.
$$
If the point $(u+\tau-t,\tau^2-t^2)$ belongs to $Q_\rho$, then, since both $t$
and $\tau$ are of magnitude $1$, then
$$
|\tau^2-t^2|\lesssim \rho^2,
\qquad\text{and hence}\qquad
|t-\tau|\lesssim \rho^2.
$$
Consequently,
$$
(u+\tau-t,\tau^2-t^2)=(u+O(\rho^2),O(\rho^2)).
$$
Thus, for such choices of $x$, the operator $A$ only sees the values of
$a_{Q_\rho}$ in a thin horizontal slab inside $Q_\rho$, namely a small
$O(\rho^2)$-neighborhood of the segment $\Sigma_{\rho}$, denoted by $N_{\rho^2} \left(\Sigma_{\rho} \right)$, rather than
the whole box $Q_\rho$ (see Figure \ref{Fig2}).

\begin{figure}[ht]
\centering
\begin{tikzpicture}[x=1.9cm,y=0.78cm,>=stealth,font=\small]

\def\rr{0.9}      
\def\ee{0.3}      
\pgfmathsetmacro{\er}{\ee*\rr}   

\pgfmathsetmacro{\slabh}{0.82*\rr*\rr}
\pgfmathsetmacro{\slabw}{\er+0.08}

\draw[->] (-1.35,0) -- (3.35,0) node[right] {$x_1$};
\draw[->] (0,-1.0) -- (0,6.95) node[above] {$x_2$};

\draw[line width=0.9pt] (-\rr,-\rr*\rr) rectangle (\rr,\rr*\rr);
\node[font=\footnotesize] at (-1.1, 0.55) {$Q_\rho$};

\fill[red!18]
(-\slabw,-\slabh) rectangle (\slabw,\slabh);
\draw[red!55!black,line width=0.7pt]
(-\slabw,-\slabh) rectangle (\slabw,\slabh);

\draw[red!85!orange,line width=1.5pt] (-\er,0) -- (\er,0);
\fill[red!85!orange] (-\er,0) circle (1.2pt);
\fill[red!85!orange] (\er,0) circle (1.2pt);
\node[red!85!orange,font=\footnotesize] at (.6,-0.45) {$\Sigma_{\rho}$};

\filldraw[black] (0.06,0) circle (1.2pt);
\node[font=\footnotesize,below] at (.65,3) {$(u,0)$};
\draw[dashed,->] (0.6,2.3) -- (0.06, 0.1);

\fill[blue!12]
plot[domain=0.5:2.45,samples=140] ({\x-\er},{\x*\x})
--
plot[domain=2.45:0.5,samples=140] ({\x+\er},{\x*\x})
-- cycle;

\draw[blue!70!black,line width=1.0pt]
plot[domain=0.5:2.45,samples=140] ({\x-\er},{\x*\x});
\draw[blue!70!black,line width=1.0pt]
plot[domain=0.5:2.45,samples=140] ({\x+\er},{\x*\x});

\draw[blue!70!black,dashed,line width=0.8pt]
plot[domain=0:2.45,samples=140] ({\x},{\x*\x});

\node[blue!70!black,font=\footnotesize,align=left]
at (2.6, 6.5)
{$\Sigma_{\rho}+\{(t,t^2):\,t\in \supp\varphi\}$};

\filldraw[black] (1.68,2.78) circle (1.5pt);
\node[font=\footnotesize,right] at (2,2.78)
{$x=(u,0)+(\tau,\tau^2)$};

\draw[dashed,->] (2.05,2.78) -- (1.73, 2.78);

\node[red!60!black,font=\scriptsize,align=left]
at (-1, 1.9)
{$N_{\rho^2} \left(\Sigma_{\rho} \right)$};

\draw[dashed,->] (-0.88,1.52) -- (-0.37, 0.1);

\end{tikzpicture}
\caption{The segment $\Sigma_{\rho}\subset Q_\rho$, and the thin horizontal
slab inside $Q_\rho$ seen by $A$.}
\label{Fig2}
\end{figure}
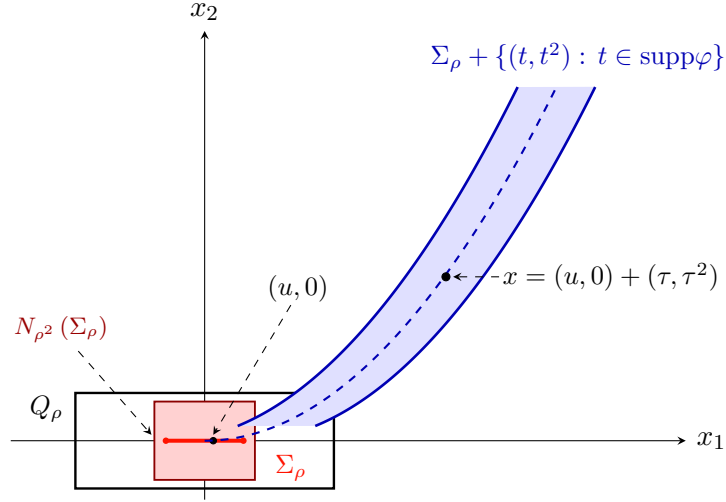

The above computation will guide the proof of Theorem~\ref{mainthm01}. First,
the admissible region for $x$ defined in \eqref{20260625comp01} may be viewed
as a thin parabolic slab of width $\simeq \eta\rho$, which is exactly the
geometric content of our first goal \eqref{20260625goal01}.

Second, once $x$ lies in the region \eqref{20260625comp01}, the preceding
incidence computation shows that only the values of $a_{Q_\rho}$ inside the
$\rho^2$-neighborhood $N_{\rho^2}\left(\Sigma_{\rho}\right)$
can contribute to $Aa_{Q_\rho}(x)$. This suggests that we should first place
a positive bump inside $N_{\rho^2}(\Sigma_{\rho})$, so that the lower
bound in \eqref{20260625goal02} holds.

Finally, this bump alone is not necessarily a $(p,N_p)$--atom, since it need
not satisfy the required moment conditions in Definition \ref{atomdefn}. We therefore add correction terms
supported in the remaining part of $Q_\rho$, away from
$N_{\rho^2}(\Sigma_{\rho})$. These correction terms are used to enforce
the moment cancellations, while being chosen so that they do not affect the
lower bound coming from the bump constructed in the previous step.

\vspace{0.1cm}

Let us now turn to the detailed argument. 

\subsection{Step I: Construction of $E_\rho$ that satisfies \eqref{20260625goal01}} \label{20260705subsec01}

Following the above heuristic, let us take a $0<\eta \ll 1$, and denote
$$
\Sigma_{\rho}:=\left\{(u, 0): |u|<\eta \rho \right\} \subset Q_\rho
$$
as above. Define now
$$
E_\rho:=\left\{(u, 0)+(\tau, \tau^2): |u|<\eta \rho, \; \tau \in \left(\frac{5}{4}, \frac{7}{4} \right) \right\}. 
$$

We estimate the size of $E_\rho$ as below. 

\begin{lem}\label{20260625lem01}
For $\rho>0$ being sufficiently small, $|E_\rho|\gtrsim \rho$.
\end{lem}

\begin{proof}
Parametrize \(E_\rho\) by
$$
        \Phi_\rho(u, \tau)=(u,0)+(\tau, \tau^2)=(u+\tau, \tau^2),
        \qquad |u|<\eta\rho,\  \tau\in \left(\frac{5}{4}, \frac{7}{4} \right).
$$
The Jacobian is
$$
        D\Phi_\rho(u,\tau)=
        \begin{pmatrix}
        1&1\\
        0&2\tau
        \end{pmatrix},
$$
and hence
$$
        |\det D\Phi_\rho(u,\tau)|=2\tau \simeq 1.
$$
As a consequence, the map $\Phi_\rho: (u, \tau) \mapsto (u+\tau, \tau^2)$ is injective on $(-\eta\rho,\eta\rho)\times (5/4, 7/4)$.  Therefore, by a change of variable, 
$$
|E_\rho|=\int_{5/4}^{7/4} \left(\int_{-\eta\rho}^{\eta\rho}du \right) 2sds
        \gtrsim \rho.
$$
The proof is complete.
\end{proof}

\subsection{Step II: Construction of a single bump near $\Sigma_{\rho}$ that satisfies \eqref{20260625goal02}} \label{20260705subsec02}

Denote 
$$
B_{0, \rho}:=(-2\eta \rho, 2\eta \rho) \times \left(-\frac{\rho^2}{2}, \frac{\rho^2}{2} \right) \subset Q_\rho.
$$
We have the following level--set estimate. 

\begin{lem} \label{20260626lem01}
For $\rho>0$ being sufficiently small and $x \in E_\rho$, one has
\begin{equation} \label{20260625eq32}
\left| \left\{t \in \left(1, 2 \right): x-(t, t^2) \in B_{0, \rho} \right\} \right| \gtrsim \rho^2.
\end{equation} 
\end{lem}

\begin{proof}
Let $x\in E_\rho$. This means that there exists $u\in\R$ and $\tau \in(5/4,7/4)$ such that
\begin{equation} \label{20260625eq31}
        x=(u,0)+(\tau, \tau^2),
\end{equation} 
with $|u|<\eta \rho$ and $\tau \in (5/4, 7/4)$. 

\vspace{0.1cm}

Choose $c_1>0$ sufficiently small so that $\frac{7}{2}c_1+c_1^2<\frac{1}{2}$. We shall also assume $\rho>0$ is sufficiently small so that
$$
        c_1\rho^2<\frac{1}{4},
        \qquad
        c_1\rho^2\le \eta\rho,
        \qquad
        \rho<1.
$$
Then, since $ \tau \in(5/4,7/4)$, we have
$$
        [\tau-c_1\rho^2, \; \tau+c_1\rho^2]\subset \left(1, 2 \right).
$$

Now take $t=\tau+h$ with
\begin{equation} \label{20260625eq30}
        |h|\le c_1\rho^2.
\end{equation} 
Then
\begin{align}  \label{20260625eq35}
x-(t, t^2)
&=(u,0)+(\tau, \tau^2)-\left(\tau+h, (\tau+h)^2 \right) \nonumber \\
&=(u-h,-2\tau h-h^2).
\end{align}
Our goal is to prove that, for every fixed $x \in E_\rho$ defined in
\eqref{20260625eq31}, one has
\begin{equation} \label{20260626eq03}
    x-(t,t^2)\in B_{0,\rho}
\end{equation} 
whenever $t=\tau+h$ with $h$ satisfying \eqref{20260625eq30}. Consequently, the bound $|h|\le c_1\rho^2$ immediately yields \eqref{20260625eq32}.

\vspace{0.1cm}

We first estimate the first coordinate in \eqref{20260625eq35}. Since $|u|<\eta\rho$ and
$|h|\le c_1\rho^2$, we have
\begin{equation} \label{20260626eq01}
|u-h| \le |u|+|h|<\eta\rho+c_1\rho^2 \le 2\eta\rho,
\end{equation} 
where we have used the fact that $\rho>0$ is sufficiently small. Next, we estimate the second coordinate in \eqref{20260625eq35}. Since $\tau \in(5/4,7/4)$, we have
$\tau<7/4$. Hence
\begin{align} \label{20260626eq02}
|-2\tau h-h^2|
&\le 2\tau |h|+|h|^2 \le \frac{7}{2} c_1\rho^2+c_1^2\rho^4 \nonumber \\
& \le \left(\frac{7}{2} c_1+c_1^2\right)\rho^2 <\frac{\rho^2}{2}.
\end{align}
Combining the estimates \eqref{20260626eq01} and \eqref{20260626eq02} yields \eqref{20260626eq03}. 

\medskip 

The proof is complete. 
\end{proof}

As an application of Lemma \ref{20260626lem01}, we construct a bump function supported near $\Sigma_{\rho}$ for which the analogue of \eqref{20260625goal02} holds, with the $(p,N_p)$-atom appearing there replaced by this bump function.
Indeed, put
\begin{equation} \label{20260626eq50}
b_{0, \rho}(x):=\rho^{-\frac{3}{p}} \one_{B_{0, \rho}}(x). 
\end{equation} 
It is clear that
\begin{enumerate}
    \item [$\bullet$] $\supp \; b_{0, \rho} \subseteq Q_\rho$;
    \item [$\bullet$] $\left\|b_{0, \rho}\right\|_{L^\infty(\R)} \le \rho^{-3/p} \simeq |Q|^{-1/p}$.
\end{enumerate}
Moreover, for any $x \in E_\rho$, one has
\begin{align} \label{20260626eq50X}
\left|Ab_{0, \rho}(x) \right| 
&=\left| \int_{\R} b_{0, \rho}(x-(t, t^2)) \varphi(t)dt \right| \nonumber \\
&=\rho^{-\frac{3}{p}} \int_{\R} \one_{B_{0, \rho}}(x-(t, t^2)) \varphi(t) dt \nonumber \\
&=\rho^{-\frac{3}{p}} \int_1^2 \one_{B_{0, \rho}}(x-(t, t^2)) dt \nonumber \\
& \gtrsim \rho^{-\frac{3}{p}} \cdot \rho^2=\rho^{2-\frac{3}{p}}, 
\end{align}
which is exactly \eqref{20260625goal02} with the parabolic $(p, \infty)$-atom there replaced by the bump function $b_{0, \rho}$. However, in general, $b_{0,\rho}$ does \emph{not} satisfy the moment conditions in Definition \ref{atomdefn}. Therefore, in the final step, we need to add suitable correction terms in order to enforce these moment conditions.

\subsection{Step III: Correction terms} \label{20260705subsec03}

The \emph{key} idea in constructing the desired $(p, \infty)$--atom is the following
observation.  For any $a$ being a parabolic $(p, \infty)$-atom and $x\in E_\rho$, only a $O(\rho^2)$-neighborhood of
$\Sigma_{\rho}$ can contribute to $Aa(x)$.  Therefore, we may place
the correction terms in the part of $Q_\rho$ which is invisible from
$E_\rho$.  These correction terms will force the required moment
conditions, but will not affect the lower bound already obtained from bump $b_{0, \rho}$ defined as in \eqref{20260626eq50}. We now turn to some details. 

\medskip 

Recall first that 
$$
E_\rho:=\left\{(u, 0)+(\tau, \tau^2): |u|<\eta \rho, \; \tau \in \left(\frac{5}{4}, \frac{7}{4} \right) \right\}. 
$$
Motivated by the inherited geometry of the single--scale operator $A$, it is natural to consider the \emph{shadow} of $E_\rho$ inside $Q_\rho$ by
\begin{align} \label{20260626shadow}
\mathcal S_\rho
&:= \left\{x-(t, t^2): x \in E_\rho, \; t \in \supp \; \varphi \right\} \cap Q_{\rho} \nonumber \\
&=\left\{(u,0)+(\tau,\tau^2)-(t,t^2): |u|<\eta\rho,\ 
        \tau\in\left(\frac54,\frac74\right),\
        t\in\supp\varphi \right\}
        \cap Q_\rho .
\end{align}
Thus, $Q_\rho\setminus\mathcal S_\rho$ does not
contribute to $Aa(x)$ for $x\in E_\rho$. 

\vspace{0.1cm}

Let $\mathcal P_{N_p}$ be the vector space of polynomials in two variables
of total degree at most $N_p$, and let
$$
        D_{N_p}:=\dim\mathcal P_{N_p}
        =
        \frac{(N_p+1)(N_p+2)}{2}.
$$
We shall use $D_{N_p}$ correction boxes to solve a $D_{N_p}$ moment
equations.

\begin{lem}
\label{20260626lem02}
Let $\eta>0$ be sufficiently small.  Then, for all sufficiently small $\rho>0$, there exist pairwise disjoint parabolic boxes
$$
        B_{1,\rho},\ldots,B_{D_{N_p},\rho}
        \subset Q_\rho\setminus\mathcal S_\rho
$$
whose side lengths are comparable to $\rho$, such that the following
holds.  Put
$$
        b_{k,\rho}:=\rho^{-\frac3p}\one_{B_{k,\rho}},
        \qquad 1\le k\le D_{N_p}.
$$
Then the moment matrix
\begin{equation}\label{20260626normmoment}
\mathfrak M_{\beta,k}:=
\rho^{-\left(3-\frac3p\right)-(\beta_1+2\beta_2)}
\int_{\R^2} x^\beta b_{k,\rho}(x)\,dx, \qquad |\beta|\le N_p,\ 1\le k\le D_{N_p},
\end{equation}
is invertible, and its inverse is bounded by a constant depending only on
$N_p$ and $\eta$, but not on $\rho$.

Moreover, there exist coefficients
$c_1,\ldots,c_{D_{N_p}}$, independent of the choice of $\rho$, such that
\begin{equation}\label{20260626momentsolve}
        \int_{\R^2}
        x^\beta
        \left(
        b_{0,\rho}
        +
        \sum_{k=1}^{D_{N_p}}c_{k}b_{k,\rho}
        \right)
        dx=0,
        \qquad |\beta|\le N_p.
\end{equation}
\end{lem}

\begin{proof}
Let $z\in\mathcal S_\rho$.  Then
$$
        z=(z_1, z_2)=(u,0)+(\tau,\tau^2)-(t,t^2)
        =
        (u+\tau-t,\tau^2-t^2),
$$
where
$$
        |u|<\eta\rho,
        \qquad
        \tau\in\left(\frac54,\frac74\right),
        \qquad
        t\in\supp \; \varphi.
$$
Since $z\in Q_\rho$, we have
$$
        |u+\tau-t|\le \rho,
        \qquad
        |\tau^2-t^2|\le \rho^2.
$$
Since $\tau\in(5/4,7/4)$ and $t\in\supp\varphi\subset(1/2,5/2)$, we have
$$
        |\tau+t|\simeq 1.
$$
Therefore the second inequality gives
$$
        |\tau-t|
        =
        \frac{|\tau^2-t^2|}{|\tau+t|}
        \lesssim \rho^2.
$$
Consequently,
$$
        |z_1|
        =
        |u+\tau-t|
        \le |u|+|\tau-t|
        \le \eta\rho+O(\rho^2)
        \le 2\eta\rho
$$
for all sufficiently small $\rho$.  Hence, after applying the anisotropic
rescaling
$$
        \delta_\rho^{-1}(x_1,x_2)
        =
        (\rho^{-1}x_1,\rho^{-2}x_2),
$$
we have
\begin{equation}\label{20260626shadowstrip}
        \delta_\rho^{-1}(\mathcal S_\rho)
        \subset
        \{(X_1, X_2)\in[-1,1]^2: |X_1|\le 2\eta\}.
\end{equation}

Now choose $\eta>0$ sufficiently small so that the open set
$$
        U:=\{(X_1, X_2): 4\eta<X_1<1/2,\ -1/2<X_2<1/2\}
$$
is non-empty.  By \eqref{20260626shadowstrip}, $U$ is disjoint from
$\delta_\rho^{-1}(\mathcal S_\rho)$ for all sufficiently small $\rho$.

Choose $D_{N_p}$ distinct points
$$
        Z_1,\ldots,Z_{D_{N_p}}\in U
$$
such that the matrix
$$
        \left[ \left(Z_k\right)^\beta\right]_{|\beta|\le N_p,\ 1\le k\le D_{N_p}}
$$
is invertible.  

Next, choose pairwise disjoint parabolic boxes
$$
        \widetilde B_1,\ldots,\widetilde B_{D_{N_p}} \subset U
$$
centered at $Z_1,\ldots,Z_{D_{N_p}}$, compactly contained in $U$, respectively, and sufficiently small. we choose these parabolic boxes to have the same fixed measure.  If the
boxes are chosen small enough, then by continuity, the matrix
\begin{equation} \label{20260626matrix01}
        \left[
        \int_{\widetilde B_k}X_1^{\beta_1}Y^{\beta_2}\,dX_1\,dX_2
        \right]_{|\beta|\le N_p,\ 1\le k\le D_{N_p}}
\end{equation} 
is arbitrarily close to a fixed non-zero multiple of
$\left[ \left(Z_k \right)^\beta\right]_{|\beta|\le N_p,\ 1\le k\le D_{N_p}}$. Hence the matrix \eqref{20260626matrix01} is invertible, and in particular, its determinant is \emph{independent} of the choice of $\rho$. 

\vspace{0.1cm}

Now define
$$
        B_{k,\rho}:=\delta_\rho(\widetilde B_k),
        \qquad 1\le k\le D_{N_p}.
$$
Then $B_{k,\rho}\subset Q_\rho\setminus\mathcal S_\rho$, the boxes are
pairwise disjoint, and their side lengths are comparable to $\rho$.
Moreover, by the change of variables
$$
        x_1=\rho X_1,
        \qquad
        x_2=\rho^2X_2,
$$
we obtain
\begin{equation} \label{20260626eq45}
        \mathfrak M_{\beta,k}=\rho^{-\left(3-\frac3p\right)-(\beta_1+2\beta_2)}
        \int_{\R^2}x^\beta b_{k,\rho}(x)\,dx
        =
        \int_{\widetilde B_k}X_1^{\beta_1}X_2^{\beta_2}\,dX_1\,dX_2.
\end{equation} 
Thus the normalized moment matrix in \eqref{20260626normmoment} is exactly the
 matrix defined as in \eqref{20260626matrix01}.  Hence it is invertible, with inverse norm
independent of $\rho$.

It remains to solve the moment equations.  Since
$$
        B_{0,\rho}
        =
        (-2\eta\rho,2\eta\rho)
        \times
        \left(-\frac{\rho^2}{2},\frac{\rho^2}{2}\right),
$$
by a similar change variable argument in \eqref{20260626eq45}, the moments
\begin{equation} \label{20260626eq46}
\rho^{-\left(3-\frac3p\right)-(\beta_1+2\beta_2)}
\int_{\R^2}x^\beta b_{0,\rho}(x)\,dx 
=\int_{(-2\eta, 2\eta) \times \left(-\frac{1}{2}, \frac{1}{2}\right)}  X_1^{\beta_1}X_2^{\beta_2} dX_1dX_2
\end{equation} 
are bounded uniformly in $\rho$.  Consider now the finite-dimensional linear system
$$
        \sum_{k=1}^{D_{N_p}}c_{k,\rho}\mathfrak M_{\beta,k}
        =
        -
        \rho^{-\left(3-\frac3p\right)-(\beta_1+2\beta_2)}
        \int_{\R^2}x^\beta b_{0,\rho}(x)\,dx,
        \qquad |\beta|\le N_p.
$$
Then the change variables \eqref{20260626eq45} and \eqref{20260626eq46} reduces the above linear system to
$$
\sum_{k=1}^{D_{N_p}} c_{k, \rho} \left(\int_{\widetilde B_k}X_1^{\beta_1}X_2^{\beta_2}\,dX_1\,dX_2 \right)=-\int_{(-2\eta, 2\eta) \times \left(-\frac{1}{2}, \frac{1}{2}\right)}  X_1^{\beta_1}X_2^{\beta_2} dX_1dX_2, \qquad |\beta| \le N_p.
$$
The above linear system is solvable since its coefficient matrix
\eqref{20260626matrix01} is invertible.  Moreover, both the coefficient matrix
and the right-hand side are independent of $\rho$.
Therefore, for all sufficiently small $\rho$, the coefficients may be chosen
independently of $\rho$.  In other words, we may write
$c_{k,\rho}=c_k$ for $1\le k\le D_{N_p}$.  The proof is complete.
\end{proof}

We now define the corrected atom.  Let $c_{1},\ldots,c_{D_{N_p}}$
be the coefficients obtained in Lemma \ref{20260626lem02}.  Set
\begin{equation}\label{20260626atomdef}
        {\mathfrak a}_{Q_\rho}(x)
        :=
        C_*
        \left(
        b_{0,\rho}(x)
        +
        \sum_{k=1}^{D_{N_p}}c_{k}b_{k,\rho}(x)
        \right),
\end{equation}
where $C_*>0$ is a sufficiently small constant depending only on $p$,
$N_p$, and $\eta$,

\begin{lem}\label{20260626lem03}
The function ${\mathfrak a}_{Q_\rho}$ is a parabolic $(p, \infty)$-atom supported in
$Q_\rho$.  Moreover, for every $x\in E_\rho$,
\begin{equation}\label{20260626lowerbound}
        A {\mathfrak a}_{Q_\rho}(x)\gtrsim \rho^{2-\frac3p}.
\end{equation}
\end{lem}

\begin{proof}
The support of $ {\mathfrak a}_{Q_\rho}$ is contained in $Q_\rho$.  The moment conditions
follow directly from \eqref{20260626momentsolve}.

We next check the $L^\infty$ normalization.  Since the coefficients
$c_{k}$ are uniformly bounded and independent of the choice of $\rho$, and since
$$
        \|b_{0,\rho}\|_\infty\le \rho^{-\frac3p},
        \qquad
        \|b_{k,\rho}\|_\infty\le \rho^{-\frac3p},
        \qquad 1\le k\le D_{N_p},
$$
we have
$$
        \|{\mathfrak a}_{Q_\rho}\|_\infty
        \lesssim
        C_*\rho^{-\frac3p}.
$$
Choosing $C_*>0$ sufficiently small gives
$$
        \|{\mathfrak a}_{Q_\rho}\|_\infty
        \le |Q_\rho|^{-1/p}.
$$
Thus $a_{Q_\rho}$ is a parabolic $(p,N_p)$-atom.

It remains to prove the lower bound.  Fix $x\in E_\rho$.  Since the
correction boxes $B_{1,\rho},\ldots,B_{D_{N_p},\rho}$ are contained in
$Q_\rho\setminus\mathcal S_\rho$, they do \emph{not} contribute to $A{\mathfrak a}_{Q_\rho}(x)$ for $x \in E_\rho$.  Therefore, following the argument in \eqref{20260626eq50X}, we deduce that 
$$
        |A{\mathfrak a}_{Q_\rho}(x)|
        =
        C_*
        \left| \int_{\R}
        b_{0,\rho}(x-(t,t^2))\varphi(t)\,dt \right| \gtrsim \rho^{2-\frac{3}{p}}.
$$
This proves \eqref{20260626lowerbound}.
\end{proof}

The proof of Theorem \ref{mainthm01} is complete. 

\medskip 

\section{The modified parabolic Hardy space $H_{\textnormal{par}}^{p, *}(\R^2)$ and the failure of $M_{\textnormal{par}}: H_{\textnormal{par}}^{p, *}(\R^2) \to L^{p, \infty}(\R^2) $} \label{Sec03}

The example constructed above in Theorem \ref{mainthm01} reveals an important
scale-dependent cost:
$$
        \left\|M_{\textnormal{par}}{\mathfrak a}_{Q_\rho}\right\|_{L^{p,\infty}(\R^2)}^p
        \gtrsim \rho^{2p-2};
$$
see \eqref{20260704eq01}. This suggests that, in order to develop a suitable
parabolic $H^p$ theory for $0<p<1$, one should incorporate this scale cost into
the definition of the parabolic $H^p$ spaces. Now we turn to some details. 

\vspace{0.1cm}

For $0<p\le 1$ and for any parabolic box $Q\subset\R^2$, define the \emph{parabolic
scale cost of $Q$} by
$$
        w_p(Q):=\bigl[\min\{1,\rho_Q\}\bigr]^{2p-2}.
$$
Here, the reason for us to use $\min\{1,\rho_Q\}$ is that in the proof of
Theorem \ref{mainthm01}, the scale $\rho$ only plays a role when it is
sufficiently small.

\begin{defn}[Modified atomic parabolic Hardy space] \label{Hparpstar}
Let $0<p\le 1$. The space $H_{\textnormal{par}}^{p, *}(\R^2)$ consists of all $f\in L_{\textrm{loc}}^1(\R^2)$
which admit a representation
\begin{equation} \label{20260705eq01}
        f=\sum_{Q} \lambda_Q a_{Q} \qquad \textrm{in $L_{\textnormal{loc}}^1(\R^2)$}
\end{equation} 
where each $a_{Q}$ is a parabolic $(p,\infty)$-atom supported in a parabolic box $Q$,
and
$$
        \sum_{Q} |\lambda_Q|^p w_p(Q)<\infty.
$$
The quasi-norm is defined by
$$
\|f\|_{H_{\textnormal{par}}^{p, *}(\R^2)}
:=\inf \left(\sum_{Q} |\lambda_Q|^p w_p(Q) \right)^{1/p},
$$
where the infimum is taken over all representations in \eqref{20260705eq01}. In particular, if $p=1$, then $H_{\textnormal{par}}^{1, *}(\R^2) \cap L^1(\R^2)$ coincides with $H_{\textnormal{par}}^1(\R^2)$.
\end{defn}

It turns out that the space $H_{\textnormal{par}}^{p,*}(\R^2)$ alone does
\emph{not} suffice to extend the parabolic $H^p$ theory to the range
$0<p<1$. There is an additional natural scaling built into the maximal operator
$M_{\textnormal{par}}$, and this scaling is reflected in the next result.

\begin{prop} \label{20260709prop01}
For every $0<p<1$, the parabolic maximal operator $M_{\textnormal{par}}$ is not bounded from $H_{\textnormal{par}}^{p, *}(\R^2)$ to $L^{p, \infty}(\R^2)$. 
\end{prop}

\begin{proof}
The proof of this result is similar to that of Theorem \ref{mainthm01}. However, we now reverse the philosophy there for the size of the parabolic
box and the support condition on $\varphi$. More precisely, now we shall consider the parabolic box of the unit size:
$$
Q_1=[-1, 1] \times [-1, 1],
$$
and the average operator 
$$
A_R f(x):=\int_{\R} f(x_1-t, x_2-t^2) \varphi \left(\frac{t}{R} \right) \frac{dt}{R}
$$
with $R \gg 1$. 

\vspace{0.1cm}

Following the construction in Theorem \ref{mainthm01}, we take a $0<\eta \ll 1$, and denote
$$
\Sigma_1:=\left\{(u, 0): |u|<\eta \right\} \subset Q_1
$$
and define
$$
E_1(R):=\left\{(u, 0)+(R\tau, R^2\tau^2): |u|<\eta, \tau \in \left(\frac{5}{4}, \frac{7}{4} \right) \right\}. 
$$
Observe first that for any parabolic $(p, \infty)$-atom $a_{Q_1}$ associated to $Q_1$, one has 
$\left\|a_{Q_1} \right\|_{H_{\textnormal{par}}^{p, *}(\R^2)} \le 1$, as the parabolic scale cost does not see large scale, that is, $w_p(Q_1)=1$.

\vspace{0.1cm}

The desired failure of boundedness of $M_{\textnormal{par}}: H_{\textnormal{par}}^{p, *}(\R^2) \to L^{p, \infty}(\R^2)$ now follows from the following two claims.
For $R\gg 1$, we have:

\medskip 

\noindent{{\bf Claim I:}} $|E_1(R)| \gtrsim R^2$.

\medskip 

\noindent{{\bf Claim II:}} There exists a parabolic $(p, \infty)$-atom ${\mathfrak a}_{Q_1}$, such that for any $x \in E_1(R)$, 
$$
\left| A_R {\mathfrak a}_{Q_1}(x) \right| \gtrsim R^{-2}. 
$$

\medskip 

Temporarily assume the above claims. Then 
\begin{align} \label{20260709eq01}
\left\|A_R{\mathfrak a}_{Q_1} \right\|_{L^{p, \infty}(\R^2)}
&=\sup_{\alpha>0} \alpha \left| \left\{x \in \R^2: \left|A_R {\mathfrak a}_{Q_1}(x) \right|>\alpha \right\} \right|^{\frac{1}{p}} \nonumber \\
& \ge R^{-2} \cdot \left|E_1(R) \right|^{\frac{1}{p}}\nonumber   \\
&\gtrsim R^{\frac{2}{p}-2}. 
\end{align}
This clearly gives the desired claim by letting $R \to \infty$.

\vspace{0.1cm}

Therefore, we are left with proving the above claims. {\bf Claim I} follows from a direct computation, using an argument similar to
the one in Section \ref{20260705subsec01}. Therefore, we focus only on {\bf Claim II}. 
Take now 
$$
x=(u, 0)+(R\tau, R^2 \tau^2) \in E_1(R), \quad \textrm{with} \quad |u|<\eta, \quad \tau \in \left(\frac{5}{4}, \frac{7}{4} \right). 
$$
Then for any parabolic $(p, \infty)$-atom $a_{Q_1}$, we have
$$
A_R a_{Q_1}(x)=\int_{\R} a_{Q_1}(x_1-t, x_2-t^2)\varphi \left(\frac{t}{R} \right) \frac{dt}{R}. 
$$
Hence, by the localization of the above operator, 
\begin{equation} \label{20260705eq10}
\left|x_1-t \right|=\left|u+R\tau-t \right|<1
\end{equation} 
and
\begin{equation} \label{20260705eq11}
\left|x_2-t^2 \right|=\left|R^2 \tau^2-t^2 \right|<1. 
\end{equation} 
Since $\tau \in (5/4, 7/4)$ and $t \simeq R$, \eqref{20260705eq11} then gives
\begin{equation} \label{20260705eq12}
\left|R\tau-t \right|=\frac{\left|R^2\tau^2-t^2\right|}{\left|R\tau+t \right|} \lesssim \frac{1}{R}.
\end{equation} 
This together with \eqref{20260705eq10} gives that the shadow of $E_1(R)$ locates in the strip
$$
\mathfrak S:=\left\{(u_1, u_2) \in Q_1: |u_1|< 2 \eta, \; |u_2|< 1 \right\}.
$$
Therefore, to prove the desired parabolic $(p, \infty)$-atom ${\mathfrak a}_{Q_1}$ satisfying {\bf Claim II}, one can follow the construction in Sections \ref{20260705subsec02} and \ref{20260705subsec03} by first construct a bump function $b_{0, 1}$ near $\mathfrak S$ and then put correction terms away from $\mathfrak S$. More precisely, write
$$
\widetilde{\mathfrak S}:=\left\{(u_1, u_2) \in Q_1: |u_1|< 2 \eta, \; |u_2|< 1/2 \right\},
$$
and define $b_{0, 1}:=\one_{\widetilde{\mathfrak S}}$. Then for any $x \in E_1(R)$, we have to estimate 
\begin{align*}
\left|A_R b_{0, 1}(x) \right|
&=\int_{R} \one_{b_{0, 1}}(x_1-t, x_2-t^2) \varphi \left(\frac{t}{R} \right) \frac{dt}{R} \\
& \ge \frac{1}{R} \cdot \int_{R}^{2R} \one_{b_{0, 1}}(x_1-t, x_2-t^2) dt \\
& \ge \frac{1}{R}  \cdot \left| \left\{ t \in [R, 2R]: (x_1-t, x_2-t^2) \in \widetilde{\mathfrak S} \right\} \right|
\end{align*}
To estimate the size of the last level set, following the argument in \eqref{20260705eq10}--\eqref{20260705eq12}, we see that for fixed $x \in E_1(R)$, one has
$$
\left| \left\{ t \in [R, 2R]: (x_1-t, x_2-t^2) \in \widetilde{\mathfrak S} \right\} \right| \gtrsim \frac{1}{R},
$$
and hence $|A_R b_{0, 1}(x)| \gtrsim 1/R^2$, which gives the correct size estimate as in {\bf Claim II}. Finally, the estimate for the correction terms follows from an argument similar
to the one in Section \ref{20260705subsec03}; we leave the details to the
interested reader.
\end{proof}

\medskip 

\section{Weak tendril spaces and an extension of Christ's theorem to the range $0<p<1$} \label{Sec04}

Proposition \ref{20260709prop01} suggests that, in order to extend Christ's result to the range $0<p<1$, one should also modify the target space $L^{p,\infty}(\R^2)$. The new target space should be adapted to the cost of each thin parabolic set $E_1(R)$, namely $R^{2-2p}$; see \eqref{20260709eq01}. To do this, we first recall the concept of tendrils introduced by Christ \cite{Christ1988}, which plays an important role in his endpoint argument. 

\subsection{Parabolic geometry and the weak tendril spaces}

For any $Q=Q(z, \rho)$ being a parabolic box and $\kappa \ge 1$, define its parabolic enlargement by 
$$
        \kappa Q(z,\rho):=Q(z,\kappa\rho)=z+[-\kappa \rho, \kappa \rho] \times [-\kappa^2 \rho^2, \kappa^2 \rho^2],
$$
and in particular, denote
$$
Q^*=4Q:=Q(z, 4\rho).
$$
Therefore, $|Q^*|=64|Q|$.

Let us recall the following very important geometric object introduced in Christ's 1988 paper \cite{Christ1988}.

\begin{defn}
Let $Q=Q(z,\rho_Q)$ be a parabolic box.  For $R\ge \rho_Q$, define the \emph{tendrils} of $Q$ by
$$
T(Q,R):=Q^*+\{(t, t^2): |t|\le R\}.
$$   
\end{defn}
An example of a tendril is shown in Figure~\ref{Fig1}.

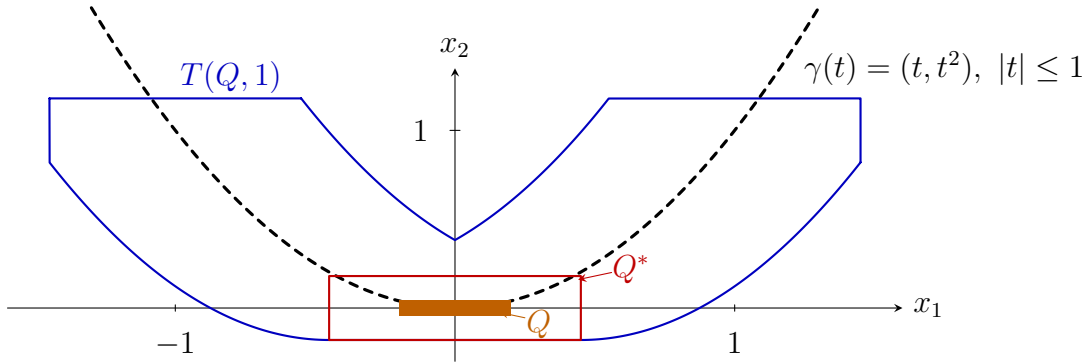
\begin{figure}[ht]
\centering
\begin{tikzpicture}[x=3.7cm,y=2.35cm,>=stealth,line cap=round,line join=round]

\draw[->] (-1.60,0) -- (1.60,0) node[right] {$x_1$};
\draw[->] (0,-0.30) -- (0,1.35) node[above] {$x_2$};

\draw (-1,0.018) -- (-1,-0.018) node[below=4pt] {$-1$};
\draw (1,0.018) -- (1,-0.018) node[below=4pt] {$1$};
\draw (0.018,1) -- (-0.018,1) node[left=4pt] {$1$};

\draw[blue!75!black, thick]
(-1.45,1.18)
-- (-0.55,1.18)
plot[domain=-0.55:0,samples=80] (\x,{(\x-0.45)*(\x-0.45)+0.18})
plot[domain=0:0.55,samples=80] (\x,{(\x+0.45)*(\x+0.45)+0.18})
-- (1.45,1.18)
-- (1.45,0.82)
plot[domain=1.45:0.45,samples=80] (\x,{(\x-0.45)*(\x-0.45)-0.18})
-- (-0.45,-0.18)
plot[domain=-0.45:-1.45,samples=80] (\x,{(\x+0.45)*(\x+0.45)-0.18})
-- (-1.45,1.18);

\draw[dashed, very thick]
plot[domain=-1.3:1.3,samples=120] (\x,{\x*\x});

\draw[red!75!black, thick] (-0.45,-0.18) rectangle (0.45,0.18);
\fill[orange!75!black] (-0.20,-0.045) rectangle (0.20,0.045);

\node[blue!75!black] at (-0.8,1.3) {$T(Q,1)$};

\node at (1.75, 1.35) {$\gamma(t)=(t,t^2),\ |t|\le 1$};

\node[red!75!black] at (0.63,0.22) {$Q^*$};
\draw[red!75!black,->] (0.56,0.20) -- (0.44,0.16);

\node[orange!75!black] at (0.30,-0.09) {$Q$};
\draw[orange!75!black,->] (0.25,-0.055) -- (0.16,-0.025);

\end{tikzpicture}
\caption{An example of tendril with $Q=Q(0, 1/10), Q^*=(0, 2/5)$, and $R=1$.}
\label{Fig1}
\end{figure}

Observe that the motivation for the above definition is clear: it captures
the localization feature of  $M_{\textnormal{par}}$. Moreover, it is only meaningful to consider the
regime in which $R$ is sufficiently larger than $\rho_Q$; otherwise, $T(Q,R)$
is contained in a fixed constant enlargement of $Q$. 

\begin{lem} \label{20260710lem01}
For every parabolic box $Q$ and every $R\ge \rho_Q$,
$$
|T(Q,R)|\simeq \rho_Q R^2.
$$
\end{lem}

\begin{proof}
We first prove the upper bound. By translation, it suffices to consider the case $Q=Q(0,\rho_Q)$, and therefore
$$
Q^*=Q(0, 4\rho_Q)=\left[-4\rho_Q, 4\rho_Q \right] \times \left[-16\rho_Q^2, 16\rho_Q^2 \right]
$$
If $x=(x_1,x_2)\in T(Q,R)$, then for some $|t|\le R$,
$$
|x_1-t|\le 4\rho_Q,
\qquad
|x_2-t^2|\le 16\rho_Q^2.
$$
In particular, the possible values of $x_1$ lie in an interval of length
$\le 5R$, since $R\ge \rho_Q$.

For each such $x_1$, set
$$
I_{x_1}:=[x_1-4\rho_Q,x_1+4\rho_Q]\cap[-R,R].
$$
Since $t^2-16\rho_Q^2 \le x_2 \le t^2+16\rho_Q^2$, the vertical section of $T(Q,R)$ at $x_1$ is contained in
$$
\{t^2:t\in I_{x_1}\}+[-16\rho_Q^2,16\rho_Q^2].
$$
Since $|I_{x_1}|\le 8\rho_Q$ and $I_{x_1}\subset[-R,R]$, the length of this
vertical section is bounded by
$$
|I_{x_1}|\sup_{|t|\le R}|2t|+32\rho_Q^2
\lesssim \rho_Q R+\rho_Q^2
\lesssim \rho_Q R,
$$
again using $R\ge \rho_Q$. Therefore, by Fubini's theorem,
$$
|T(Q,R)| \le \int_{|x_1| \le 5R} \left(\int_{I_{x_1}} \one_{T(Q, R)}(x_1, x_2) dx_2 \right) dx_1 \lesssim R\cdot \rho_Q R=\rho_Q R^2.
$$

\medskip

For the lower bound, by translation we may again assume that
$Q=Q(0,\rho_Q)$. Since $Q^*=Q(0,4\rho_Q)$, we have
$$
\{(u,0): |u|\le \eta \rho_Q\}\subset Q^*
$$
for some sufficiently small absolute constant $\eta>0$. Hence
$$
S:=\{(u,0)+(\tau,\tau^2): R/2\le \tau\le R,\ |u|\le \eta\rho_Q\}
        \subset T(Q,R).
$$
Using the same change-of-variables argument as in Lemma \ref{20260625lem01}, we obtain
$$
|T(Q,R)|\ge |S|
=
\int_{R/2}^{R}\int_{-\eta\rho_Q}^{\eta\rho_Q} 2\tau\,du\,d\tau
\gtrsim \rho_QR^2.
$$
This gives the desired lower bound.
\end{proof}

\medskip

The argument in Proposition \ref{20260709prop01} suggests that the definition of weak $L^{p,\infty}(\R^2)$ must be modified to reflect the geometry of tendrils.

\begin{defn}[Tendril cover]
Let $E\subset\R^2$ be any measurable set.  A \emph{tendril cover} of $E$ is a countable family
$$
        \{ \left(\theta_i, T(Q_i,R_i) \right)\}_{i \ge 1}, \quad \textrm{with}  \qquad \theta_i \ge 0, 
        \quad R_i \ge \rho_{Q_i},
$$
such that
$$
\one_E(x)\leq
 \sum_{i\geq1}\theta_i\one_{T(Q_i,R_i)}(x)
 \quad\text{for almost every }x\in\R^2.
$$
\end{defn}

\begin{defn}[Weak tendril spaces] \label{20260714defn01}
Let $0<p\le 1$, and let $E\subseteq \R^2$ be a measurable set. Define the
\emph{$p$-tendril content} of $E$ by
\begin{equation} \label{20260710eq01}
        C_p(E)
        :=
        \inf
        \sum_{i\ge 1} \theta_i \rho_{Q_i}R_i^{2p},
\end{equation}
where the infimum is taken over all countable tendril coverings of $E$.
The \emph{weak tendril space} $\calT^{p,\infty}(\R^2)$ is defined to be the
collection of all measurable functions $f$ on $\R^2$ such that
$$
        \|f\|_{\calT^{p,\infty}(\R^2)}^p
        :=
        \sup_{\alpha>0}\alpha^p C_p(\{|f|>\alpha\})<\infty.
$$
\end{defn}

By Lemma \ref{20260710lem01}, we have
$$
C_p(E)
\simeq
\inf \sum_{i\ge 1} \theta_i |T(Q_i,R_i)|\,R_i^{2p-2}.
$$
Here the factor $R_i^{2p-2}$ is precisely the correction term dictated by the
cost appearing in \eqref{20260709eq01}.

\medskip

We first show that the $p$-tendril content is non-degenerate. 

\begin{lem} \label{v27:lem:local}
Let $0<p \le 1$, and let $K\subset\R^2$ be a fixed bounded rectangle. Then
\begin{equation}\label{v27:eq:local}
        |T(Q,R)\cap K|
        \lesssim_{p,K}
        \rho_QR^{2p}
\end{equation}
for every parabolic box $Q$ and every $R\geq\rho_Q$.
\end{lem}

\begin{proof}
If $R\leq1$, then Lemma~4.2 gives
$$
|T(Q,R)\cap K| \leq |T(Q,R)|
\lesssim \rho_QR^2 \leq \rho_QR^{2p}.
$$
If $\rho_Q\geq1$, then $R\geq\rho_Q\geq1$, and hence
$$
        |T(Q,R)\cap K|
        \leq
        |K|
        \lesssim_K
        \rho_QR^{2p}.
$$
Finally, suppose that $R>1$ and $\rho_Q<1$. Write $Q=Q(z,\rho_Q)$, where
$z=(z_1,z_2)$. For each $x_2 \in\R$, define
$$
        A_{x_2}
        :=
        \left\{
        t\in[-R,R]:
        |x_2-z_2-t^2|\leq16\rho_Q^2
        \right\}.
$$
By the definition, for any $(x_1, x_2)\in T(Q,R)$, there exists $t\in A_{x_2}$ such that
$$
        |x_1-z_1-t|
        \leq
        4\rho_Q.
$$
Consequently, the horizontal section of $T(Q,R)$ at height $x_2$ is
contained in
$$
        z_1+A_{x_2}+[-4\rho_Q,4\rho_Q].
$$

We claim that
\begin{equation} \label{20260712claim01}
|A_{x_2}| \lesssim \rho_Q
\end{equation} 
uniformly in $x_2$. We consider three cases. 

\medskip 

\noindent $\bullet$ If $x_2-z_2<-16\rho_Q^2$, then it is clear that $A_{x_2}$ is empty. 

\medskip 

\noindent $\bullet$ If $-16\rho_Q^2
\leq x_2-z_2 \leq 16\rho_Q^2$, 
then every $t\in A_{x_2}$ satisfies
$$
|t|\leq\sqrt{x_2-z_2+16\rho_Q^2} \leq 4\sqrt{2}\rho_Q,
$$
and therefore $|A_{x_2}|\lesssim\rho_Q$.

\medskip 

\noindent $\bullet$ If $x_2-z_2>16\rho_Q^2$, since the set determined by
$$
|t^2-(x_2-z_2)| \leq 16\rho_Q^2
$$
is the union of the two intervals
$$
        \left[
        \sqrt{x_2-z_2-16\rho_Q^2}, 
        \; 
        \sqrt{x_2-z_2+16\rho_Q^2}
        \right]
\quad \textrm{and} \quad 
        \left[
        -\sqrt{x_2-z_2+16\rho_Q^2}, \; 
        -\sqrt{x_2-z_2-16\rho_Q^2}
        \right].
$$
The sum of their lengths equals
$$
\begin{aligned}
        2\left(
        \sqrt{x_2-z_2+16\rho_Q^2}
        -
        \sqrt{x_2-z_2-16\rho_Q^2}
        \right)
        &=
        \frac{64\rho_Q^2}
        {\sqrt{x_2-z_2+16\rho_Q^2}
        +\sqrt{x_2-z_2-16\rho_Q^2}}  \\
        &\lesssim
        \rho_Q,
\end{aligned}
$$
since $x_2-z_2>16\rho_Q^2$. Intersecting these intervals with $[-R,R]$
cannot increase their total length. Therefore, claim \eqref{20260712claim01} holds. 

\vspace{0.1cm}

As a consequence, we have 
$$
        \left|
        z_1+A_{x_2}+[-4\rho_Q,4\rho_Q]
        \right|
        \lesssim
        \rho_Q.
$$
Therefore, by Fubini, we have 
$$
\begin{aligned}
        |T(Q,R)\cap K|
        &=
        \int_{\pi_2(K)}
        \left|
        \left\{
        x_1:(x_1, x_2)\in T(Q,R)\cap K
        \right\}
        \right|\,dx_2  \\
        &\le \int_{\pi_2(K)} \left|
        z_1+A_{x_2}+[-4\rho_Q,4\rho_Q]
        \right| dx_2 \\ 
        &\lesssim
        |\pi_2(K)|\cdot \rho_Q
        \lesssim_K
        \rho_Q \le_K \rho_Q R^{2p} 
\end{aligned}
$$
where  $\pi_2(K)$ denotes the vertical
projection of $K$, and in the last estimate above, we have used the assumption $R>1$. 

\vspace{0.1cm}

The proof  \eqref{v27:eq:local} is complete. 
\end{proof}

\begin{prop} \label{20260714prop10} 
Let $0<p<1$. If $E\subset\R^2$ has positive Lebesgue measure, then
$$
        C_p(E)>0.
$$
\end{prop}

\begin{proof}
Choose a bounded rectangle $K$ such that $|E\cap K|>0$. If
$\{(\theta_i,T(Q_i,R_i))\}_i$ is a tendril cover of $E$,
then by
Lemma~\ref{v27:lem:local}, we obtain
\begin{align*} 
|E\cap K|
&=\int_K \one_E(x) dx \nonumber \\
&\le \int_K \left(\sum_{i \ge 1} \theta_i \one_{T(Q_i, R_i)}(x) \right)dx \nonumber  \\
&\leq \sum_i\theta_i|T(Q_i,R_i)\cap K| \nonumber  \\
        &\lesssim_{p,K}
        \sum_i\theta_i\rho_{Q_i}R_i^{2p}.
\end{align*}
Taking the infimum over all tendril covers of $E$ gives
\begin{equation} \label{20260713eq01}
        C_p(E)
        \gtrsim_{p,K}
        |E\cap K|
        >0.
\end{equation} 
\end{proof}

Next, we observe that the weak tendril spaces $\mathcal T^{p,\infty}(\mathbb R^2)$ provide a natural extension of $L^{1,\infty}(\mathbb R^2)$, adapted to the underlying parabolic geometry.

\begin{prop}
\label{20260710prop01}
One has
$$
        \calT^{1,\infty}(\R^2)=L^{1,\infty}(\R^2)
$$
with equivalent quasi-norms. 
\end{prop}

\begin{proof}
It suffices to show that 
\begin{equation} \label{20260710eq10}
        C_1(E)\simeq |E|.
\end{equation} 
We begin with the lower bound. Let
$\{(\theta_i, \; T(Q_i,R_i))\}_{i\ge 1}$ be any tendril cover of $E$. Hence, by Lemma \ref{20260710lem01},
$$
        |E|
        \le
        \sum_{i\ge 1} \theta_i |T(Q_i,R_i)|
        \lesssim
        \sum_{i\ge 1} \theta_i \rho_{Q_i}R_i^2.
$$
Taking the infimum over all tendril covers of $E$ yileds the desired lower bound. 

\vspace{0.1cm}

Next, we prove the upper bound in \eqref{20260710eq10}. We may assume that $|E|<\infty$, otherwise there is nothing to prove. Fix $\varepsilon>0$. Take an open set $O\subset \R^2$ such that
$$
        E\subset O,
        \qquad
        |O|\le |E|+\varepsilon.
$$

We now cover $O$ by dyadic parabolic boxes. More precisely, for
$k\in\Z$, let $\mathcal P_k$ denote the collection of all dyadic parabolic boxes
$$
P:=[m2^{-k},(m+1)2^{-k})\times [n2^{-2k},(n+1)2^{-2k}), \qquad m,n\in\Z.
$$
Each such box has horizontal length $2^{-k}$, vertical length $2^{-2k}$,
and area $|P|=2^{-3k}$. Let $\mathcal P(O)$ be the collection of maximal dyadic parabolic rectangles
contained in $O$. Since $O$ is open, these maximal rectangles cover $O$ up to a
null set. Moreover, they are pairwise disjoint. Hence
$$
        \sum_{P \in\mathcal P(O)} |P| \le |O|.
$$
For each $P\in\mathcal P(O)$, write its horizontal length as $\ell(P)$
and let $z_P$ be its center. Define $Q_P:=Q(z_P,\ell(P))$. Note that 
$$
P\subset Q_P\subset Q_P^*\subset T(Q_P,\ell(P)).
$$
Thus the collection
$$
        \{1, \; T(Q_P,\ell(P))\}_{P\in\mathcal P(O)}
$$
is a tendril cover of $O$, and hence also of $E$. Consequently,
\begin{align*}
C_1(E)
&\le \sum_{P\in\mathcal P(O)} \rho_{Q_P}\ell(P)^2 =\sum_{P\in\mathcal P(O)} \ell(P)^3 \\
&=\sum_{P\in\mathcal P(O)} |P|  \le |O| \le |E|+\varepsilon.
\end{align*}
Letting $\varepsilon\to 0$, we get the desired upper bound $C_1(E)\lesssim |E|$.
\end{proof}

To this end, we show that the $\calT^{p, \infty}(\R^2)$ does not remove the small scale obstruction identified in Theorem \ref{mainthm01}. 

\begin{prop}\label{v27:prop:failure}
For every $0<p<1$, the parabolic maximal operator $M_{\textnormal{par}}$ is
not bounded from $H_{\textnormal{par}}^p(\R^2)$ to
$\calT^{p,\infty}(\R^2)$.
\end{prop}

\begin{proof}
Let ${\mathfrak a}_{Q_\rho}$ and $E_\rho$ be the parabolic $(p, \infty)$-atom and the exceptional set constructed in Section \ref{Sec02}, which satisfy
$$
\|{\mathfrak a}_{Q_\rho}\|_{H_{\textnormal{par}}^p(\R^2)}
\lesssim 1, \qquad |E_\rho| \gtrsim \rho,
$$
and
$$ 
M_{\textnormal{par}}{\mathfrak a}_{Q_\rho}(x)
\gtrsim \rho^{2-\frac3p}, \qquad x\in E_\rho.
$$
Moreover, since $E_\rho$ is contained in $[0, 10]^2$, by \eqref{20260713eq01}, we have 
$$
        C_p(E_\rho)
        \gtrsim_p
        |E_\rho \cap [0, 10]^2|
        \gtrsim
        \rho.
$$
Consequently,
$$
\begin{aligned}
        \|M_{\textnormal{par}}{\mathfrak a}_{Q_\rho}
        \|_{\calT^{p,\infty}(\R^2)}^p
        &\gtrsim
        \left(\rho^{2-\frac3p}\right)^p
        C_p(E_\rho)  \\
        &\gtrsim
        \rho^{2p-3}\rho
        =
        \rho^{2p-2},
\end{aligned}
$$
which converges to $\infty$ as $\rho \to 0$.
\end{proof}

\subsection{An extension of Christ's theorem to the range $0<p<1$}

We now prove the positive result. Recall that
$$
        w_p(Q)
        =
        \bigl[\min\{1,\rho_Q\}\bigr]^{2p-2}.
$$
For a parabolic box $Q=Q(z,\rho_Q)$ and $x\in\R^2$, define its
\emph{tendril radius} by
$$
        R_Q(x)
        :=
        \inf\{R\geq\rho_Q:x\in T(Q,R)\},
$$
with the convention that $R_Q(x)=\infty$ when the set on the right-hand
side is empty.

\begin{lem}\label{20260713lem01}
Let $0<p\leq1$, and let $a_Q$ be a parabolic $(p,\infty)$-atom supported
in $Q=Q(z,\rho_Q)$. Then
$$
M_{\textnormal{par}}a_Q(x) \lesssim \rho_Q^{2-\frac3p}R_Q(x)^{-2}, \qquad x\in\R^2.
$$
\end{lem}

\begin{proof}
Fix $r>0$ and $x\in\R^2$, and set
$$
I_{x,r}:=\left\{ t\in\supp\;\varphi\left(\frac{\cdot}{r} \right) \subseteq \left(\frac{r}{2}, \frac{5r}{2} \right): \;  x-(t,t^2)\in Q \right\}.
$$
By the size condition of a parabolic $(p, \infty)$-atom, 
\begin{equation}\label{20260713eq10}
|A_ra_Q(x)|
= \left| \int_{\R} a_Q \left(x_1-t, x_2-t^2 \right) \varphi \left(\frac{t}{r} \right) \frac{dt}{r} \right| \lesssim \rho_Q^{-\frac3p} \cdot \frac{|I_{x,r}|}{r}.
\end{equation}
Here, we may assume $I_{x,r} \neq \varnothing$, otherwise there is nothing to prove. 

\vspace{0.1cm}

Suppose first that $R_Q(x)\le 100\rho_Q$. Since
$R_Q(x)\geq\rho_Q$, it follows that $R_Q(x)\simeq\rho_Q$. Then \eqref{20260713eq10} together with the trivial
estimate $|I_{x,r}|\lesssim r$ then gives
$$
        |A_ra_Q(x)|
        \lesssim
        \rho_Q^{-\frac3p}
        \simeq
        \rho_Q^{2-\frac3p}R_Q(x)^{-2}.
$$

Next, we consider the case when $R_Q(x)>100\rho_Q$. Choose
$t_0\in I_{x,r}$, and hence $t_0 \simeq r$ and $x \in Q+(t_0,t_0^2)$. Observe that in this case one has $t_0 \ge 20\rho_Q$, otherwise $x \in Q+(t_0, t_0^2) \subseteq T(Q, 25\rho_Q)$. This implies $R_Q(x) \le 25\rho_Q$, which contradicts to our assumption. Therefore, 
$$
x \in Q+(t_0, t_0^2) \subseteq T(Q, t_0),
$$
which gives $R_Q(x)\leq t_0$.  On the other hand, if
$x\in T(Q,R)$, then
$$
        x
        =
        z+u+(t_0,t_0^2)
        =
        z+v+(s,s^2)
$$
for some $u\in Q-z$, $v\in Q^*-z$, and $|s|\leq R$. Comparing the first
coordinates gives
$$
        |t_0-s|
        \le
        5\rho_Q,
$$
and hence
$$
        R
        \geq
        t_0-5\rho_Q. 
$$
Taking the infimum over all such $R$ such that $x \in T(Q, R)$ gives
$$
        R_Q(x)
        \geq
        t_0-5\rho_Q. 
$$
Therefore,
$$
r \simeq t_0 \simeq R_Q(x).
$$
For every $t\in I_{x,r}$, both $x-(t,t^2)$ and
$x-(t_0,t_0^2)$ belong to $Q$. Comparing the second coordinates gives
$$
        |t^2-t_0^2|
        \lesssim
        \rho_Q^2.
$$
Since $t+t_0\simeq r \simeq  R_Q(x)$, we obtain
$$
        |t-t_0|
        \lesssim
        \frac{\rho_Q^2}{R_Q(x)},
        \qquad \textrm{and hence} \qquad 
        |I_{x,r}|
        \lesssim
        \frac{\rho_Q^2}{R_Q(x)}.
$$
Substituting these estimates into \eqref{20260713eq10} and using
$r\simeq R_Q(x)$, we obtain
$$
        |A_ra_Q(x)|
        \lesssim
        \rho_Q^{2-\frac3p}R_Q(x)^{-2},
$$
where the implict constant in the above estimate is indepednent of $r$.
Taking the supremum over $r>0$ proves the lemma.
\end{proof}

\begin{prop}\label{20260714prop01}
Let $0<p \le 1$, and let $a_Q$ be a parabolic $(p,\infty)$-atom supported
in $Q$. Then
\begin{equation}\label{v27:eq:singleatom}
        \|M_{\textnormal{par}}a_Q\|_{\calT^{p,\infty}(\R^2)}^p
        \lesssim_p
        w_p(Q).
\end{equation}
\end{prop}

\begin{proof}
For $\alpha>0$, set
$$
E_\alpha:=
\{M_{\textnormal{par}}a_Q>\alpha\}.
$$
Without loss of generality, we may assume $E_\alpha$ is not empty. Then for any $x_0 \in E_\alpha$, by Lemma~\ref{20260713lem01},  
\begin{equation} \label{20260714eq10}
\alpha<C\rho_Q^{2-\frac{3}{p}}R_Q(x_0)^{-2} \le C\rho_Q^{2-\frac{3}{p}} \cdot \rho_Q^{-2} 
\end{equation} 
for some absolute constant $C>0$, where we have used the fact that $R_Q(x_0) \ge \rho_Q$. This means $x_0 \in T(Q, R_\alpha)$, where 
$$
R_\alpha:=\left(\frac{C\rho_Q^{2-\frac3p}}{\alpha} \right)^{1/2}
>\rho_Q.
$$
Otherwise $R_Q(x_0)>R_\alpha$, which contradicts with \eqref{20260714eq10}.
Therefore, we derive that 
$$
        E_\alpha
        \subset
        T(Q,R_\alpha).
$$
Hence
$$
        C_p(E_\alpha) \le C_p \left(T(Q, R_{\alpha}) \right)
        \leq
        \rho_Q R_\alpha^{2p}
        \lesssim
        \alpha^{-p}\rho_Q^{2p-2}.
$$
If $0<\rho<1$, then
$$
\rho_Q^{2p-2}=w_p(Q),
$$
while if $\rho_Q \geq1$, then
$$
\rho_Q^{2p-2} \leq 1
=w_p(Q).
$$
Therefore,
$$
        \alpha^pC_p(E_\alpha)
        \lesssim_p
        w_p(Q).
$$
Taking the supremum over $\alpha>0$ proves
\eqref{v27:eq:singleatom}.
\end{proof}

The following crucial lemma shows that, for $0<p<1$, the
$\calT^{p,\infty}$ quasi-norm has a subadditive structure. We emphasize
that the assumption $p<1$ is essential for this result.

\begin{lem}\label{20260714prop22}
Let $0<p<1$, and let $\{F_j\}_{j \in J}$ be a finite or countable family of nonnegative measurable
functions. Suppose that there exist numbers $A_j\geq0$ such that $\sum_{j \in J} A_j<\infty$ and 
\begin{equation}\label{20260714eq30}
C_p(\{F_j>t\}) \leq A_jt^{-p}, \qquad t>0.
\end{equation}
Then
\begin{equation}\label{20260714eq31}
        \left\|
        \sum_{j \in J} F_j
        \right\|_{\calT^{p,\infty}(\R^2)}^p
        \lesssim_p
        \sum_{j \in J} A_j.
\end{equation}
\end{lem}

\begin{proof}
For a nonnegative measurable function $G$, define its \emph{$p$-tendril covering cost} by
$$
\Lambda_p(G):=\inf \left\{
\sum_{i \ge 1} \theta_i\rho_{Q_i}R_i^{2p}:
G(x) \le \sum_{i \ge 1} \theta_i\one_{T(Q_i,R_i)}(x), \; a.e. \; x \in \R^2, \quad  \theta_i \ge 0, R_i \ge \rho_{Q_i}
\right\}.
$$
We record several basic properties of $\Lambda_p$.

\begin{enumerate}
    \item [$\bullet$] For every measurable set $E\subseteq\R^2$,
    \begin{equation} \label{20260714eq49}
            \Lambda_p(\one_E)
            =
            C_p(E).
    \end{equation} 
    This follows immediately from the definition of $C_p(E)$; see
    Definition~\ref{20260714defn01}.

    \item [$\bullet$]  $\Lambda_p$ is positively homogeneous. More
    precisely, $\Lambda_p(0)=0$, and for every $c>0$ and every nonnegative
    measurable function $G$,
    \begin{equation} \label{20260714eq50}
\Lambda_p(cG)=c\Lambda_p(G).
    \end{equation}

    \item [$\bullet$] $\Lambda_p$ is monotone, in the sense that if $0 \le G \le H$, then $\Lambda_p(G) \le \Lambda_p(H)$.  

    \item [$\bullet$] $\Lambda_p$ is countably subadditive in the
    following sense. Let $G$ and $\{G_m\}_{m\geq1}$ be nonnegative
    measurable functions such that
   $$
            G(x)
            =
            \sum_{m=1}^{\infty}G_m(x)
 $$
    for almost every $x\in\R^2$. Then
     \begin{equation} \label{20260714eq52}
            \Lambda_p(G)
            \leq
            \sum_{m=1}^{\infty}\Lambda_p(G_m).
     \end{equation}
This follows clearly from the definition. 
\end{enumerate}

\vspace{0.1cm}

Fix now $\alpha>0$ and set
$$
        G_j
        :=
        F_j\one_{\{F_j\leq \frac{\alpha}{2}\}}.
$$
The pointwise dyadic decomposition
$$
G_j=\sum_{k \ge 0} F_j \one_{\left\{\frac{\alpha}{2^{k+2}}<F_j \le  \frac{\alpha}{2^{k+1}} \right\}}
        \lesssim
        \sum_{k\geq0}
        \frac{\alpha}{2^k}\one_{\left\{F_j>\frac{\alpha}{2^{k+2}}\right\}},
$$
together with \eqref{20260714eq49}, \eqref{20260714eq50}, \eqref{20260714eq52}, and \eqref{20260714eq30} gives
\begin{align} \label{20260714eq54}
        \Lambda_p(G_j)
        & \le \sum_{k \ge 0} \Lambda_p \left( \frac{\alpha}{2^k}\one_{\left\{F_j>\frac{\alpha}{2^{k+2}}\right\}} \right)  \nonumber \\ 
        &\lesssim
        \sum_{k\geq0}
       \frac{\alpha}{2^k} \cdot 
        C_p \left(\left\{F_j>\frac{\alpha}{2^{k+2}}\right\} \right) \nonumber   \\
        &\lesssim_p
        A_j\alpha^{1-p}
        \sum_{k\geq0}2^{-k(1-p)}  \nonumber \\
        &\lesssim_p
        A_j\alpha^{1-p},
\end{align}
where we have used the assumption $p<1$ in the last estimate above. 

Now put
$$
F:=\sum_{j \in J} F_j, \qquad \textrm{and} \qquad G:= \sum_{j \in J} G_j.
$$
We claim that
\begin{equation}\label{20260714eq40}
 \{F>\alpha\}
 \subseteq
 \left(\bigcup_{j\in J}\{F_j>\alpha/2\}\right)
 \cup\{G>\alpha\}.
\end{equation}
Indeed, if $x \in \{F>\alpha\} \backslash \left(\bigcup_{j\in J}\{F_j>\alpha/2\}\right)$, then $F_j(x) \le \alpha/2$, and hence $F_j(x)=G_j(x)$, which further yields $F(x)=G(x)$. Since $F(x)>\alpha$, this forces $G(x)>\alpha$. The claim \eqref{20260714eq40} is proved.

By the countably subadditivity of $\Lambda_p$, one has
\begin{equation} \label{20260714eq25}
C_p \left( \left\{F>\alpha\right\} \right) \le C_p \left(\bigcup_{j\in J}\{F_j>\alpha/2\}\right) +C_p \left( \{G>\alpha\} \right)
\end{equation} 
To estimate the right hand side of \eqref{20260714eq25}, first using the countably subadditivity of $\Lambda_p$ again with the assumption \eqref{20260714eq30}, we see that
\begin{equation} \label{20260714eq26}
        C_p\left(
        \bigcup_{j \in J} \{F_j>\alpha/2\}
        \right) \le \sum_{j \in J} C_p \left(\{F_j>\alpha/2\}
        \right) 
        \lesssim
        \alpha^{-p}
        \sum_{j \in J} A_j.
\end{equation} 
Next, to estimate the second $p$-tendril content in \eqref{20260714eq26}, using the fact that $\one_{G>\alpha} \le G/\alpha$ and \eqref{20260714eq54}, we have 
\begin{align} \label{20260714eq26a}
        C_p(\{G>\alpha\})
        &=\Lambda_p(\one_{G>\alpha}) \le \Lambda_p \left( \frac{G}{\alpha} \right) \nonumber \\
        &=\alpha^{-1}\Lambda_p(G) \le 
        \alpha^{-1}
        \sum_{j \in J} \Lambda_p(G_j) \nonumber  \\
        &\lesssim_p
        \alpha^{-p}
        \sum_{j \in J} A_j.
\end{align} 
Combining \eqref{20260714eq25}, \eqref{20260714eq26}, and \eqref{20260714eq26a} yields 
$$
        C_p(\{F>\alpha\})
        \lesssim_p
        \alpha^{-p}
        \sum_{j \in J} A_j.
$$
Multiplying by $\alpha^p$ and taking the supremum over $\alpha>0$
proves \eqref{20260714eq31}.
\end{proof}

We are ready to state our main result in this section. 

\begin{thm} \label{mainwhole}
Let $0<p<1$. Then $M_{\textnormal{par}}$ extends to a bounded operator
$$
        M_{\textnormal{par}}:
        H_{\textnormal{par}}^{p,*}(\R^2)
        \longrightarrow
        \calT^{p,\infty}(\R^2).
$$
More precisely, let
$$
f=\sum_{j\in J}\lambda_ja_{Q_j}
$$
be a finite or countable atomic representation, and assume
that
$$
        \sum_{j\in J}|\lambda_j|^pw_p(Q_j)<\infty.
$$
Then
\begin{equation}\label{20260714main01}
        \|M_{\textnormal{par}}f\|_{\calT^{p,\infty}(\R^2)}^p
        \lesssim_p
        \sum_{j\in J}|\lambda_j|^pw_p(Q_j).
\end{equation}
\end{thm}

\begin{proof}
We divide the proof into several steps. 

\vspace{0.1cm}

\noindent{\bf Step 1.}
For each $j\in J$, set
$$
        F_j
        :=
        |\lambda_j|M_{\textnormal{par}}a_{Q_j},
$$
and define
$$
        G
        :=
        \sum_{j\in J}F_j.
$$
By Proposition~\ref{20260714prop01}, there exists a constant
$c_p>0$ such that, for every $t>0$,
$$
        C_p(\{F_j>t\})
        \leq
        c_p|\lambda_j|^pw_p(Q_j)t^{-p}.
$$
Applying Lemma~\ref{20260714prop22} with
$$
        A_j
        :=
        c_p|\lambda_j|^pw_p(Q_j),
$$
we obtain
\begin{equation}\label{20260714main02}
        \|G\|_{\calT^{p,\infty}(\R^2)}^p
        \lesssim_p
        \sum_{j\in J}|\lambda_j|^pw_p(Q_j).
\end{equation}

\medskip

\noindent{\bf Step 2.}
We claim that
\begin{equation}\label{20260714claim01}
        G(x)<\infty
        \qquad
        \textrm{for almost every }x\in\R^2.
\end{equation}
Indeed, suppose that
$$
        E_\infty
        :=
        \{x\in\R^2:G(x)=\infty\}
$$
has positive Lebesgue measure. By Proposition~\ref{20260714prop10},
$$
        C_p(E_\infty)>0.
$$
Since
$$
        E_\infty
        \subseteq
        \{G>\alpha\}
$$
for every $\alpha>0$, we have
$$
\begin{aligned}
        \|G\|_{\calT^{p,\infty}(\R^2)}^p
        &\geq
        \alpha^pC_p(\{G>\alpha\})\\
        &\geq
        \alpha^pC_p(E_\infty).
\end{aligned}
$$
Letting $\alpha\to\infty$ contradicts \eqref{20260714main02}.
Therefore, \eqref{20260714claim01} holds.

\medskip

\noindent{\bf Step 3.}
We prove that
\begin{equation}\label{20260714pointwise}
        M_{\textnormal{par}}f(x)
        \leq
        G(x)
\end{equation}
for almost every $x\in\R^2$.

If $J$ is finite, \eqref{20260714pointwise} follows immediately from
the sublinearity of $M_{\textnormal{par}}$.

Suppose that $J$ is countable. After relabeling, we may assume that
$J=\mathbb N$. For $N\geq1$, set the partial sum $
f_N:=\sum_{j=1}^N\lambda_ja_{Q_j}$. 
By \eqref{20260714claim01}, 
$$
        \sum_{j=N+1}^{\infty}F_j(x)
        \longrightarrow0
        \qquad\textrm{as }N\to\infty
$$
for almost every $x\in\R^2$. Moreover, for every $M>N$,
\begin{equation} \label{20260715eq01}
        \sup_{r>0}
        |A_rf_M(x)-A_rf_N(x)|
        \leq
        \sum_{j=N+1}^MF_j(x).
\end{equation} 
Thus, for almost every $x\in\R^2$, the sequence
$\{A_rf_N(x)\}_{N\geq1}$ is uniformly Cauchy with respect to $r>0$.

Since $f_N \to f \; \textrm{in }L_{\textrm{loc}}^1(\R^2)$, 
by Fubini, we have 
$$
        A_rf_N
        \longrightarrow
        A_rf
        \qquad
        \textrm{in }L_{\textrm{loc}}^1(\R^2)
$$
for every fixed $r>0$. Letting now $M \to \infty$ in \eqref{20260715eq01} and arguing via a standard limiting argument, we derive that 
$$
        M_{\textnormal{par}}f_N(x)
        \longrightarrow
        M_{\textnormal{par}}f(x)
$$
for almost every $x\in\R^2$.

For every $N\geq1$, it is clear that $M_{\textnormal{par}}f_N \leq \sum_{j=1}^NF_j$. 
Letting $N\to\infty$, we obtain
$M_{\textnormal{par}}f \leq \sum_{j=1}^{\infty}F_j=G$
almost everywhere. This proves \eqref{20260714pointwise}.

\medskip

\noindent{\bf Step 4.}
By \eqref{20260714pointwise}, the monotonicity of the
$\calT^{p,\infty}(\R^2)$ quasi-norm, and \eqref{20260714main02},
$$
\begin{aligned}
        \|M_{\textnormal{par}}f\|_{\calT^{p,\infty}(\R^2)}^p
        &\leq
        \|G\|_{\calT^{p,\infty}(\R^2)}^p\\
        &\lesssim_p
        \sum_{j\in J}|\lambda_j|^pw_p(Q_j).
\end{aligned}
$$
This proves \eqref{20260714main01}. Taking the infimum over all atomic
representations of $f$ gives
$$
        \|M_{\textnormal{par}}f\|_{\calT^{p,\infty}(\R^2)}
        \lesssim_p
        \|f\|_{H_{\textnormal{par}}^{p,*}(\R^2)}.
$$
The proof is complete.
\end{proof}

\bigskip 

\section{Further remarks} \label{Sec05}

We conclude with some observations concerning interpolation of these new function spaces.
It is natural to ask whether the modified Hardy spaces 
$H_{\textnormal{par}}^{p,*}(\R^2)$ and weak tendril spaces 
$\calT^{p,\infty}(\R^2)$ introduced in this paper admit interpolation 
properties compatible with the $L^2$ boundedness of 
$M_{\textnormal{par}}$. Our first result shows that standard real interpolation of these new function spaces with $L^2$ cannot simultaneously recover the standard parabolic Hardy space and
the weak $L^q$ space for $q<1$.

\vspace{0.1cm}

Here and henceforth, for a pair  of quasi-Banach spaces $(X_0,X_1)$, 
$0<\theta<1$, and $0<r\leq\infty$, we denote by 
$(X_0,X_1)_{\theta,r}$ the real interpolation space obtained by the $K$-method. 
More precisely, for $f\in X_0+X_1$ and $t>0$, define
$$
K(t,f;X_0,X_1):=\inf_{f=f_0+f_1}
\left(\|f_0\|_{X_0}+t\|f_1\|_{X_1}\right),
$$
where the infimum is taken over all decompositions $f=f_0+f_1$ with 
$f_0\in X_0$ and $f_1\in X_1$. If $0<r<\infty$, set
$$
\|f\|_{(X_0,X_1)_{\theta,r}}
:=\left(\int_0^\infty \bigl[t^{-\theta}K(t,f;X_0,X_1)\bigr]^r
\frac{dt}{t}\right)^{1/r},
$$
while for $r=\infty$,
$$
\|f\|_{(X_0,X_1)_{\theta,\infty}} :=\sup_{t>0}t^{-\theta}K(t,f;X_0,X_1).
$$
The space $(X_0,X_1)_{\theta,r}$ consists of all $f\in X_0+X_1$ for which the corresponding quantity above is finite.

\begin{prop}
Let $0<p_0<q<1$ and $0<r\leq\infty$, and let $0<\theta<1$ satisfy
$$
\frac1q=\frac{1-\theta}{p_0}+\frac{\theta}{2}.
$$
Then the two continuous embeddings
$$
H_{\textnormal{par}}^q(\mathbb R^2) \hookrightarrow \bigl(H_{\textnormal{par}}^{p_0,\ast}(\mathbb R^2),L^2(\mathbb R^2)\bigr)_{\theta,r}
$$
and
$$
\bigl(\mathcal T^{p_0,\infty}(\mathbb R^2),L^2(\mathbb R^2)\bigr)_{\theta,r} \hookrightarrow L^{q,\infty}(\mathbb R^2)
$$
cannot hold simultaneously.
\end{prop}

\begin{proof}
Recall that $A$ is the single-scale averaging operator in Theorem~\ref{mainthm01}. Since $|Af|\leq M_{\textnormal{par}}f$, Theorem~\ref{mainwholeX} implies that
$A:H_{\textnormal{par}}^{p_0,\ast}(\mathbb R^2)\longrightarrow\mathcal T^{p_0,\infty}(\mathbb R^2)$ is bounded. Moreover, it is also clear that $\|Af\|_{L^2(\R^2)}\leq\|\varphi\|_{L^1(\R)}\|f\|_{L^2(\R^2)}$. By real interpolation, we conclude that
$$
A:\bigl(H_{\textnormal{par}}^{p_0,\ast}(\R^2),L^2(\R^2) \bigr)_{\theta,r}\longrightarrow\bigl(\mathcal T^{p_0,\infty}(\R^2),L^2(\R^2) \bigr)_{\theta,r}
$$
is bounded. If both embeddings held, we would therefore obtain
$$
\|Af\|_{L^{q,\infty}(\R^2)}\lesssim\|f\|_{H_{\textnormal{par}}^q(\R^2)}.
$$
This contradicts Theorem~\ref{mainthm01}.
\end{proof}

Our next result shows that real interpolation between these new spaces
does \emph{not}, in general, recover the corresponding spaces at the
intermediate exponent, even when both endpoint exponents lie below $1$.

\begin{lem}
Let $0<p_0<p_1\leq1$, $0<\theta<1$, and $0<r\leq\infty$. Define $p$ by
$$
\frac1p=\frac{1-\theta}{p_0}+\frac{\theta}{p_1}.
$$
Then
\begin{equation} \label{20260918eq01A}
\bigl(H_{\textnormal{par}}^{p_0,*}(\R^2),H_{\textnormal{par}}^{p_1,*}(\R^2)\bigr)_{\theta,r}
\not\hookrightarrow H_{\textnormal{par}}^{p,*}(\R^2)
\end{equation}
and
\begin{equation} \label{20260918eq01B}
\bigl(\calT^{p_0,\infty}(\R^2),\calT^{p_1,\infty}(\R^2)\bigr)_{\theta,r}
\not\hookrightarrow \calT^{p,\infty}(\R^2).
\end{equation} 
\end{lem}

\begin{proof}
We first prove the assertion for the modified Hardy spaces $H_{\textnormal{par}}^{p,*}(\R^2)$. Choose an integer $N\geq1$ sufficiently large and a nonzero function $\psi\in C_c^\infty(\R^2)$ such that
$$
\supp\psi\subset Q(0,1/10), \qquad \|\psi\|_{L^\infty(\R^2)}\leq1,
$$
and
$$
\int_{\R^2}x^\beta\psi(x)\,dx=0, \qquad |\beta|\leq N.
$$

For an integer $L\geq2$, define
$$
z_{ij}:=\left(\frac{i-\frac12}{L},\frac{j-\frac12}{\lfloor L^{2p}\rfloor}\right),
\qquad 1\leq i\leq L,\quad 1\leq j\leq \lfloor L^{2p}\rfloor.
$$
Let
$$
Q_{ij}:=Q\left(z_{ij},\frac1{10L}\right),
$$
and set
$$
\psi_{ij}(x):=\psi\bigl(L(x_1-\left(z_{ij}\right)_1),L^2(x_2-\left(z_{ij} \right)_2 )\bigr),
\qquad
f_L:=\sum_{i=1}^L\sum_{j=1}^{\lfloor L^{2p}\rfloor}\psi_{ij}.
$$
Observe that the parabolic boxes $Q_{ij}$ are pairwise disjoint and contained in $[0,1]^2$. Moreover,
$$
\supp\psi_{ij}\subset Q_{ij}, \qquad \|f_L\|_{L^\infty(\R^2)}\leq1,
$$
and every $\psi_{ij}$, as well as $f_L$, has vanishing moments through order $N$.

\vspace{0.1cm}

We claim that
\begin{equation}\label{20260918eq01}
\|f_L\|_{\left(H_{\textnormal{par}}^{p_0,*}(\R^2),H_{\textnormal{par}}^{p_1,*}(\R^2)\right)_{\theta,r}}
\lesssim L^{-2\theta(1-p/p_1)}
\end{equation}
and
\begin{equation}\label{20260918eq02}
\|f_L\|_{H_{\textnormal{par}}^{p,*}(\R^2)}\gtrsim1.
\end{equation}
Once these two estimates are established, the desired conclusion  \eqref{20260918eq01A} follows immediately by passing $L \to \infty$. We now prove \eqref{20260918eq01} and \eqref{20260918eq02} separately.

\medskip

\noindent{\bf Proof of \eqref{20260918eq01}.}
Observe first that
\begin{equation}\label{20260918eq03}
\|f_L\|_{H_{\textnormal{par}}^{p_0,*}(\R^2)}\lesssim1.
\end{equation}

Moreover, since $|Q_{ij}|^{-1/p_1}\psi_{ij}$ is a parabolic $(p_1,\infty)$-atom, we also have
\begin{equation}\label{20260918eq04}
\|f_L\|_{H_{\textnormal{par}}^{p_1,*}(\R^2)}^{p_1} \leq \sum_{i=1}^L\sum_{j=1}^{\lfloor L^{2p}\rfloor}|Q_{ij}|w_{p_1}(Q_{ij}) \lesssim L\lfloor L^{2p}\rfloor L^{-1-2p_1} \lesssim L^{-2(p_1-p)},
\end{equation}
which gives 
$$
\|f_L\|_{H_{\textnormal{par}}^{p_1,*}(\R^2)}
\lesssim L^{-2(1-p/p_1)}.
$$

Hence, by \eqref{20260918eq03} and \eqref{20260918eq04} and real interpolation, we obtain
$$
\|f_L\|_{\left(H_{\textnormal{par}}^{p_0,*}(\R^2),
H_{\textnormal{par}}^{p_1,*}(\R^2)\right)_{\theta,r}}
\lesssim L^{-2\theta(1-p/p_1)},
$$
which proves \eqref{20260918eq01}.

\medskip

\noindent{\bf Proof of \eqref{20260918eq02}.}
Set
$$
U_L:=\bigcup_{i=1}^L\bigcup_{j=1}^{\lfloor L^{2p}\rfloor}Q_{ij}.
$$
We first claim that, for every parabolic box $Q$,
\begin{equation}\label{20260918eq05}
L^{2-2p}|Q\cap U_L|\lesssim |Q|w_p(Q).
\end{equation}

Indeed, if $0<\rho_Q\leq L^{-1}$, then
$$
L^{2-2p}|Q\cap U_L|
\leq L^{2-2p}|Q|
\lesssim L^{2-2p}\rho_Q^3
\leq \rho_Q^{1+2p}.
$$
If $L^{-1}\leq\rho_Q\leq1$, then $Q$ intersects at most $CL\rho_Q$ columns and
$C(1+\lfloor L^{2p}\rfloor\rho_Q^2)$ rows of the parabolic boxes $Q_{ij}$. Since
$|Q_{ij}|\simeq L^{-3}$, we obtain
\begin{align*}
L^{2-2p}|Q\cap U_L|
&\lesssim L^{2-2p}L\rho_Q
\bigl(1+\lfloor L^{2p}\rfloor\rho_Q^2\bigr)L^{-3}\\
&\lesssim \rho_QL^{-2p}+\rho_Q^3
\lesssim \rho_Q^{1+2p}.
\end{align*}
Finally, if $\rho_Q\geq1$, then
$$
L^{2-2p}|Q\cap U_L|
\leq L^{2-2p}|U_L|
\lesssim1
\lesssim\rho_Q^3.
$$
Since
$$
|Q|w_p(Q)\simeq
\begin{cases}
\rho_Q^{1+2p},&0<\rho_Q\leq1,\\
\rho_Q^3,&\rho_Q\geq1,
\end{cases}
$$
the claim \eqref{20260918eq05} follows.

Now for any $f \in H_{\textnormal{par}}^{p, *}(\R^2)$, let $f=\sum_Q\lambda_Qa_Q$ be any atomic representation as in
Definition~\ref{Hparpstar}. For each parabolic $(p,\infty)$-atom $a_Q$, by
\eqref{20260918eq05},
$$
L^{2-2p}\int_{U_L}|a_Q|^p\,dx
\leq L^{2-2p}|Q|^{-1}|Q\cap U_L|
\lesssim w_p(Q), 
$$
which, together with fact that $0<p<1$, implies 
$$
L^{2-2p}\int_{U_L}|f|^p\,dx
\lesssim\sum_Q|\lambda_Q|^pw_p(Q).
$$
Taking the infimum over all atomic representations for $f$ gives
\begin{equation}\label{20260918eq06X}
L^{2-2p}\int_{U_L}|f|^p\,dx
\lesssim \|f\|_{H_{\textnormal{par}}^{p,*}(\R^2)}^p.
\end{equation}

Applying \eqref{20260918eq06X} to $f_L$, and using the disjointness of the
boxes $Q_{ij}$, we obtain
\begin{align*}
\|f_L\|^p_{H_{\textnormal{par}}^{p,*}} 
& \gtrsim L^{2-2p}\int_{U_L}|f_L|^p\,dx \\
&=L^{2-2p}\sum_{i=1}^L\sum_{j=1}^{\lfloor L^{2p}\rfloor}
\int_{\R^2}|\psi_{ij}(x)|^p\,dx\\
&=\frac{\lfloor L^{2p}\rfloor}{L^{2p}}
\int_{\R^2}|\psi(x)|^p\,dx
\simeq 1,
\end{align*}
which proves \eqref{20260918eq02}.

\medskip

We next prove the assertion for the weak tendril spaces $\calT^{p,\infty}(\R^2)$, whose argument follows a similar strategy to that used in the first part. Again, let $L\geq2$ be a large integer. Define
$$
E_L:=[0,L]\times\bigcup_{j=0}^{\lfloor L^{2p}\rfloor-1}
\left[
\frac{jL^2}{\lfloor L^{2p}\rfloor},
\frac{jL^2}{\lfloor L^{2p}\rfloor}+1
\right].
$$
Observe first that the horizontal strips in the definition of $E_L$ are pairwisely disjoint, and hence
$$
|E_L|=L\lfloor L^{2p}\rfloor\simeq L^{1+2p}.
$$

We claim that
\begin{equation}\label{20260918eq07}
\|\one_{E_L}\|_{\left(\calT^{p_0,\infty}(\R^2),
\calT^{p_1,\infty}(\R^2)\right)_{\theta,r}}
\lesssim L^{2+1/p-2\theta(1-p/p_1)}
\end{equation}
and
\begin{equation}\label{20260918eq08}
\|\one_{E_L}\|_{\calT^{p,\infty}(\R^2)}
\gtrsim L^{2+1/p}.
\end{equation}
Assuming temporarily that both \eqref{20260918eq07} and \eqref{20260918eq08} hold, we show that the desired conclusion \eqref{20260918eq01B} holds. Indeed, since $p<p_1$, we have $2\theta(1-p/p_1)>0$, and hence
$$
\frac{
\|\one_{E_L}\|_{\calT^{p,\infty}(\R^2)}
}{
\|\one_{E_L}\|_{\left(\calT^{p_0,\infty}(\R^2),
\calT^{p_1,\infty}(\R^2)\right)_{\theta,r}}
}
\gtrsim L^{2\theta(1-p/p_1)}\to\infty
\qquad \textrm{as} \qquad L\to\infty, 
$$
This gives $\bigl(\calT^{p_0,\infty}(\R^2),\calT^{p_1,\infty}(\R^2)\bigr)_{\theta,r}
\not\hookrightarrow \calT^{p,\infty}(\R^2)$. Therefore, we are left to prove \eqref{20260918eq07} and \eqref{20260918eq08}.

\medskip

\noindent{\bf Proof of \eqref{20260918eq07}.} First, since
$E_L\subset Q(0,L)\subset T(Q(0,L),L)$, 
we have
$$
C_{p_0}(E_L)\lesssim L^{1+2p_0}.
$$
Therefore,
\begin{equation}\label{20260918eq09}
\|\one_{E_L}\|_{\calT^{p_0,\infty}(\R^2)}
\lesssim L^{2+1/p_0}.
\end{equation}

On the other hand, each horizontal strip in $E_L$ can be covered by at most $ \lesssim L$ parabolic boxes of radius $1$. Replacing each such box $Q$ by the tendril $T(Q,1)$, we see that 
$$
C_{p_1}(E_L)\lesssim L\lfloor L^{2p}\rfloor
\lesssim L^{1+2p}.
$$
Hence,
\begin{equation}\label{20260918eq10}
\|\one_{E_L}\|_{\calT^{p_1,\infty}(\R^2)}
\lesssim L^{(1+2p)/p_1}.
\end{equation}

By \eqref{20260918eq09}, \eqref{20260918eq10}, and real interpolation,
$$
\|\one_{E_L}\|_{\left(\calT^{p_0,\infty}(\R^2),
\calT^{p_1,\infty}(\R^2)\right)_{\theta,r}}
\lesssim
L^{(1-\theta)(2+1/p_0)+\theta(1+2p)/p_1}= L^{2+1/p-2\theta(1-p/p_1)}.
$$
This proves \eqref{20260918eq07}.

\medskip

\noindent{\bf Proof of \eqref{20260918eq08}.}
We first observe that for every bounded interval $I\subset\R$, one has
\begin{equation}\label{20260918eq11}
|\pi_2(E_L)\cap I|\lesssim_p |I|^p,
\end{equation}
where $\pi_2$ denotes the projection onto the second coordinate.

Indeed, if $|I|\leq1$, then
$$
|\pi_2(E_L)\cap I|\leq |I|\leq |I|^p.
$$
If $1<|I|\leq L^2$, then the number of unit intervals in $\pi_2(E_L)$ which intersect $I$ is at most
$$
C\left(1+\frac{\lfloor L^{2p}\rfloor}{L^2}|I|\right).
$$
Therefore,
$$
|\pi_2(E_L)\cap I|
\lesssim 1+L^{2p-2}|I|
\lesssim |I|^p.
$$
Finally, if $|I|>L^2$, then
$$
|\pi_2(E_L)\cap I|
\leq \lfloor L^{2p}\rfloor
\lesssim L^{2p}
\leq |I|^p.
$$
Thus \eqref{20260918eq11} holds.

\vspace{0.1cm}

Now let $T(Q,R)$ be any tendril with $R\geq\rho_Q$. By 
\eqref{20260712claim01}, every horizontal section of $T(Q,R)$ has length at most $C\rho_Q$. Moreover, if $Q=Q(z,\rho_Q)$, then
$$
\pi_2(T(Q,R))
\subseteq
[z_2-16\rho_Q^2,z_2+R^2+16\rho_Q^2],
$$
and hence $\pi_2(T(Q,R))$ is contained in an interval of length at most of magnitude $R^2$. It follows from Fubini's theorem and \eqref{20260918eq11} that
\begin{equation}\label{20260918eq12}
|T(Q,R)\cap E_L|
\lesssim \rho_QR^{2p}.
\end{equation}

Let $\{(\theta_i,T(Q_i,R_i))\}_i$ be any tendril cover of $E_L$. By \eqref{20260918eq12}, we obtain
$$
|E_L|
\leq \sum_i\theta_i|T(Q_i,R_i)\cap E_L|
\lesssim \sum_i\theta_i\rho_{Q_i}R_i^{2p}.
$$
Taking the infimum over all such tendril covers, we deduce that 
$$
C_p(E_L)\gtrsim |E_L|
\simeq L^{1+2p}.
$$
Therefore,
$$
\|\one_{E_L}\|_{\calT^{p,\infty}(\R^2)}
=C_p(E_L)^{1/p}
\gtrsim L^{(1+2p)/p}
=L^{2+1/p},
$$
which proves \eqref{20260918eq08} and completes the proof.
\end{proof}

\end{document}